\documentclass[a4paper,11pt]{amsart}

\usepackage[latin1]{inputenc}
\usepackage[T1]{fontenc}
\usepackage[english]{babel}

\usepackage{amsmath,amssymb,amsthm,enumerate,xspace}
\usepackage{comment}
\usepackage{accents}
\usepackage[colorlinks=true,citecolor=blue,linkcolor=blue]{hyperref}
\usepackage[all]{xy}
\usepackage{tikz}
\usepackage{mathrsfs}
\usepackage{cleveref}

\usepackage{varwidth}

\usepackage{mathrsfs}
\usepackage{amscd}
\usepackage{amsfonts}
\usepackage{amsmath}
\usepackage{amssymb}
\usepackage{amstext}
\usepackage{amsthm}
\usepackage{amsbsy}

\usepackage{xspace}
\usepackage[all]{xy}
\usepackage{graphicx}
\usepackage{url}
\usepackage{latexsym}

\usepackage{cleveref}
\usepackage{tikz-cd}

\newtheorem{theo}{Theorem}[section]
\newtheorem{lem}[theo]{Lemma}
\newtheorem{prop}[theo]{Proposition}
\newtheorem{cor}[theo]{Corollary}

\newtheorem{claim}[theo]{Claim}

\newtheorem{thmintro}{Theorem}

\newtheorem{corintro}[thmintro]{Corollary}
\newtheorem{propintro}[thmintro]{Proposition}

\newtheorem{innercustomgeneric}{\customgenericname}
\providecommand{\customgenericname}{}
\newcommand{\newcustomtheorem}[2]{%
  \newenvironment{#1}[1]
  {%
   \renewcommand\customgenericname{#2}%
   \renewcommand\theinnercustomgeneric{##1}%
   \innercustomgeneric
  }
  {\endinnercustomgeneric}
}

\newcustomtheorem{customthm}{Theorem}
\newcustomtheorem{customprop}{Proposition}
\newcustomtheorem{customlemma}{Lemma}
\newcustomtheorem{customcorollary}{Corollary}

 \theoremstyle{definition}

\newtheorem{ex}[theo]{Example}

\newtheorem{fact}[theo]{Fact}
\newtheorem{rem}[theo]{Remark}

\newcommand{\N}{\ensuremath{\mathbb{N}}}  
\newcommand{\Z}{\ensuremath{\mathbb{Z}}}

\newcommand{\R}{\ensuremath{\mathbb{R}}}
\newcommand{\C}{\ensuremath{\mathbb{C}}}

\newcommand{\M}{\ensuremath{\mathcal{M}}}
\newcommand{\Rs}{\ensuremath{\mathcal{R}}}

\newcommand{\rng}{\ensuremath{\mathcal{R} = (R, +, 0, \cdot)}}

\newcommand{\mtf}{\ensuremath{\mathcal{N}}}

\newcommand{\df}{\rangle_\mathrm{def}}

\newcommand{\inv}{^{-1}}

\newcommand{\ol}{\overline}
\newcommand{\la}{\langle}

\newcommand{\se}{\subseteq}

\DeclareMathOperator{\Ann}{Ann}

\renewcommand{\leq}{\leqslant}
\renewcommand{\geq}{\geqslant}
\begin{document}

 \title{Ring theory in o-minimal structures}
 \author{Annalisa Conversano}

 \begin{abstract}
We develop a general ring theory in the o-minimal setting culminating in a description of all the definable rings  
in an arbitrary o-minimal structure. 

We show that every definably connected ring with non-trivial multiplication defines an infinite field and it is essentially semialgebraic. 

A surprisingly strong correspondence between definably connected rings and finite-dimensional associative $\R$-algebras is established. 
 Every ideal of a definable unital ring is definable, from which it follows that every definable unital ring is Artinian and Noetherian. If a definable ring $R$ is not unital, we give necessary and sufficient conditions for $R$ to embed in a definable unital ring as an ideal. Moreover, when this is the case, we provide the smallest such definable unital ring $R^{\wedge}$, its definable 
unitazation.
\end{abstract}

\address{Massey University Auckland, New Zealand} 
\email{a.conversano@massey.ac.nz}
  
\noindent
\date{\today}

\maketitle
\tableofcontents

\section{Introduction}

 By a \textbf{ring} we mean an abelian group with an associative multiplication that distributes over the addition on the left and 
on the right. In particular, rings are not assumed to be unital nor commutative.

\begin{thmintro}\label{theo:fields}
Let \M\  
%$\M = (M, <, +, \cdot, \dots)$ 
be an o-minimal expansion of a field. The \M-definable rings are:

\begin{enumerate}[(a)]

\medskip
\item the finite-dimensional associative $\M$-algebras;

\medskip
\item the finite rings embedded in a definable abelian group with a product annihilating the definably connected component of zero;

\medskip
\item the direct products of rings
\[
A \times B 
\]

\medskip \noindent
where $A$ and $B$ are isomorphic to rings as in $(a)$, $(b)$ respectively.

\end{enumerate}
 \end{thmintro}

\newpage
If the structure does not expand a field, we prove a field is definable nonetheless.

\begin{thmintro}\label{theo:field}
Every definably connected ring that is not a null ring defines an infinite field.
\end{thmintro}

It follows that an o-minimal structure \M\ defines (interprets) an infinite field if and only if \M\ defines (interprets) a definably connected ring with non-trivial multiplication.

 \medskip
From now on let \M\ be an arbitrary o-minimal structure. Unless otherwise stated, {\bf definable} means ``definable in \M\ with parameters". When we say that a set is {\bf definably connected}, we assume it is definable.

\bigskip
  \begin{thmintro}\label{theo:Jacobson}
Let $R$ be a $n$-dimensional definably connected ring and $J(R)$ its Jacobson radical. The following are equivalent:

\begin{enumerate} 
\item $R$ is a nilpotent ring.

\item $R$ is a nil ring.

\item $R^{n+1} = \{0\}$.

\item $R = J(R)$. 

%\item $R$ is a Jacobson ring.

\item The only idempotent element in $R$ is zero. 
\end{enumerate}

If $R$ is not nilpotent, $J(R)$ is the maximal nilpotent \Rs-definable ideal, it contains every nilpotent (definable or not) ideal of $R$, the quotient $R/J(R)$ is a semiprime ring, and there is a definable subring $S$ which is definably isomorphic to $R/J(R)$ such that $R = J(R) \oplus S$.  
\end{thmintro}

%-----------------------
\bigskip
\begin{thmintro}\label{theo:semiprime}
 Let $R$ be a definably connected ring. The following are equivalent:
 \begin{enumerate}[(i)]

\item $R$ is semiprime.

\item $R$ is semisimple.

 \item $J(R) = \{0\}$.

%\item $\soc(R) = \soc_L(R) = \soc_R(R) = R$.

 \item $R$ is a direct product of simple definable rings. 
 
 \item $R$ is unital and either $R$ is a division ring or there is a (unique) finite set of primitive orthogonal idempotents $\{e_1, \dots, e_n\}$ such that $1 = \sum_{i = 1}^n {e_i}$ and
 \[
 R = \bigoplus_{i = 1}^n Re_i  = \bigoplus_{i = 1}^n e_iR  = \bigoplus_{i, j = 1}^n e_iRe_j,
 \]
 
 \medskip \noindent
 where $\{Re_i: i = 1, \dots n\}$ is the set of minimal left ideals of $R$ and $\{e_iR: i = 1, \dots n\}$ is the set of minimal right ideals of $R$.  Moreover, each $e_iRe_j$ is an infinite subring that is a division ring when $i = j$ and it is a null ring when $i \neq j$.

\end{enumerate}

 \end{thmintro}

\bigskip
 \begin{thmintro}\label{theo:simple}
Let $R$ be an infinite definable ring with non-trivial multiplication. The following are equivalent:

\begin{enumerate} [(a)] 

%\medskip
\item $R$ is simple.

\item $R$ is definably simple.

%\medskip
\item $R$ is definably connected and prime.

%\medskip
\item There is a definable real closed field $K$ and a definably connected $K$-definable  division ring $D$ such that $R$ is definably isomorphic to $M_n(D)$, for some $n \geq 1$.
  \end{enumerate}
\end{thmintro}

%--------------------
\newpage
 \begin{thmintro}\label{theo:main}
Let $R$ be a definably connected ring. If $R$ is not a null ring, there is a positive integer $s$ and definable real closed fields $K_1, \dots, K_s$ such that 

\begin{enumerate}
\medskip
\item
$K_i$ is not definably isomorphic to $K_j$, if $i \neq j$;

\medskip
\item
$R$ is a direct product of rings 
\[
R = R_0 \times R_1 \times \cdots \times R_s
\]

\medskip \noindent
where $R_0 \subseteq \Ann(R)$ - and therefore $R^2 \se R_1 \times \cdots \times R_s$ -  and for $i > 0$, $R_i$ is a finite dimensional associative $K_i$-algebra that is not a null ring;

\medskip
\item if the Jacobson radical $J$ is a null ring or $\mtf(J, +)$ is a direct sum of definable $1$-dimensional subgroups, then for 
$i > 0$, $R_i$ is a definable $K_i$-vector space and a definable $K_i$-algebra;

\medskip
\item
if $R$ is unital, then 
\[
R(1) = K_1 \times \cdots \times K_s, 
\]
\medskip
$(R, +)$ is torsion-free and a direct sum of $1$-dimensional definable subgroups;   

\medskip
\item
if $R$ is not unital, there is a unital definable ring containing $R$ as an ideal
if and only if $(R_0, +)$ is definable torsion-free and $(J, +)$ is a direct sum of definable $1$-dimensional subgroups 
\[
 J = \bigoplus_{i = 1}^{\dim J} A_i,
\]

\medskip
where each $(A_i, +)$ admits a definable multiplication making it a real closed field. 

\noindent
When this is the case, there is a smallest definably connected unital ring $R^{\wedge}$ containing $R$ as an ideal, and $\dim R^{\wedge} - \dim R \leq  \dim R_0 + s$. We call $R^{\wedge}$ the \textbf{definable unitazation} of $R$.
\end{enumerate}
\end{thmintro}

\medskip
When a definable ring is not definably connected, one can reduce to finite rings and definably connected rings due to the following:

  \begin{propintro}\label{prop:disconnected}
 Let \rng\ be a  definable ring. Set $G = (R, +)$. Suppose $R$ is not definably connected and $|R/R^0| = n$. The $n$-torsion subgroup $T_n$ of $G$ is a finite ideal such that $R^2 \subseteq \mtf(G) \times T_n$.  If $R$ is unital,  $R$ is a direct product of unital rings      
 \[
 R = F \times R^0
 \]
 
 \medskip
 where $F$ is a finite ideal and $(R^0, +)$ is torsion-free.
\end{propintro}

 \medskip
 \begin{corintro}
If $R$ is a definable ring, $\mtf(R, +)$ is elementarily equivalent to a $n$-dimensional associative $\R$-algebra, where $n = \dim \mtf(R, +)$.
\end{corintro}

\medskip
\begin{corintro}
Let $R$ be a definable ring. If $\mtf(R, +)$  is a direct sum of definable $1$-dimensional subgroups, then $R/\Ann(R)$ is Artinian, Noetherian and every ideal is definable.
In particular, if $R$ is unital, every ideal is definable and $R$  is Artinian and Noetherian.
\end{corintro}

\bigskip
The paper is organized as follows: In Section 2 we introduce the terminology and related notation, and we prove some general results we will be using throughout the paper. In Section 3 nilpotent groups are studied and \Cref{theo:main} is proved for them in \Cref{theo:nilpotent}. Section 4 describes definable reduced rings as a direct product of division rings that are definable in the ring structure. In Section 5 the Jacobson radical is introduced and \Cref{theo:Jacobson} is proved. Section 6 is devoted to semiprime rings and the proofs of \Cref{theo:semiprime} and \Cref{theo:simple}. Finally, the general case is treated in Section 7 where the proofs of \Cref{theo:fields}, \Cref{theo:field}, \Cref{theo:main} and \Cref{prop:disconnected} are completed.
  
\medskip
\textbf{Acknowledgments.} Thanks to Will Johnson for correspondence during the preparation of this manuscript.

%================================  
%%===============================

\section{Some general results}

 Recall that by \cite{OPP96}, any definable ring $\Rs = (R, +, 0, \cdot)$ admits a definable ring-manifold structure, and if  $\M$ is an o-minimal expansion of a real closed field, then $\Rs$ admits a definable $k$-differentiable ring-manifold for any natural number $k$.  
 
 We say that a ring $R$ has \textbf{trivial multiplication} when $xy = 0$ for each $x, y \in R$.   Such ring is also called a \textbf{null ring}. In other words, $R$ is a null ring if and only if $R^2 = \{0\}$.
 
 A note about the power notation. In ring theory usually $R^k$ denotes the ideal generated by the products of $k$ elements from $R$. However, it will be convenient for us to use this notation to denote just the products, as the ideal generated by a definable set is not generally definable. That is, in this paper,
\[
R^{k+1} = \bigcup_{a \in R}aR^{k} = \bigcup_{a \in R}R^ka,
\]

where, for a set $X \subseteq R$ and $a \in R$, as usual, 
\[
Xa = \{xa : x \in X\}\quad \mbox{ and }\quad  aX = \{ax: x \in X\}.
\]

Clearly $R^{k}$ is a definable set. Each $aR^{k}$ contains $0$, so if $R$ is definably connected, then $R^k$ is definably connected for each $k$, by the same argument showing that unions of connected sets with non-trivial intersection is connected.

%\medskip
The \textbf{center} $Z(R)$ of a ring $R$ is the set of elements commuting with every element in $R$:
\[
Z(R) = \{a \in R: ax = xa \ \forall\, x \in R\}.
\]

\medskip
The center is a subring, but not necessarily an ideal. For example, the center of $M_n(\R)$, the ring of $n \times n$ real matrices, is the subring of diagonal matrices, which is not an ideal.

Whenever the ring is unital, its center is of course not trivial.

For each $a \in R$, the \textbf{centralizer $C(a)$ of $a$} is  
\[
C(a) = \{r \in R: ra = ra\}.
\] 

%\medskip
 $C(a)$ is a definable subring containing $a$. Its center $Z(C(a))$ is a commutative definable subring containing $a$. Because of DCC on definable subgroups, there is a smallest definable subrgroup $\la X \df$ of $(R, +)$ containing $X$ for each subset $X$. It follows that there is also a \textbf{smallest definable subring of $R$ containing $X$}. We denote it by 
 $\bf{R(X)}$. 
By the above, $R(X)$  is commutative whenever the elements in $X$ commute pairwise and 
$R(X) \subseteq C(X)$. In particular, note that $R(a)$ is commutative for each $a \in R$.
In any case, $C(X)$ is a definable subring by DDC again.

For $X \subseteq R$ define its \textbf{annihilators} to be
\[
\Ann_1(X) = \{a \in R: ax = 0\ \ \forall x \in X\}, \quad \Ann_2(X) = \{a \in R: xa = 0\ \ \forall x \in X\}
\]
\[
\Ann(X) = \Ann_1(X) \cap \Ann_2(X) = \{a \in R: ax = xa = 0\ \ \forall x \in X\}.
\]

\medskip
For each $X$, $\Ann_1(X)$ is a definable left ideal, $\Ann_2(X)$ is a definable right ideal, $\Ann(X)$ is a definable subring. Note that by DCC on definable subgroups we can assume $X$ is finite, so they are all $\Rs$-definable. When $X = R$ , $\Ann_1(R)$, $\Ann_2(R)$ and $\Ann(R)$ 
are \Rs-definable ideals.  
 
If $R$ is any ring, an element $x \in R$, $x \neq 0$ is called a \textbf{left zero-divisor} when $\Ann_2(x) \neq \{0\}$ and a  
 \textbf{right zero-divisor} when $\Ann_1(x) \neq \{0\}$.  When we say that $x$ is a  \textbf{zero-divisor} we mean that $x$ is a left or right zero-divisor. In general, a left zero-divisor does not need to be a right zero-divisor, and viceversa.

%\medskip
\begin{fact} \label{fact:finite-torsion}
Fix $n \in \N$. Suppose $G$ is a definable group such that $g^n = e$ for all  $g \in G$. Then $G$ is finite. 
\end{fact}

\begin{cor}
Every infinite definable ring has characteristic zero.
\end{cor}

 Other results about groups we will be using are the following:

\begin{fact}\label{fact:divisible-split}
Let $(G, +)$ be an abelian group. If $A$ is a divisible subgroup, there is a subgroup $B$ such that
\[
G = A \oplus B.
\]
\end{fact}

\begin{fact}\label{fact:tf-divisible}  
 Let $G$ be a torsion-free definable group. For each positive integer $n$ the definable map $G \to G$ given by $g \mapsto g^n$ is a bijection. That is, $G$ is uniquely divisible.
\end{fact}

\begin{fact} \label{fact:1-dim} \label{fact:1dimtf}  \cite{PS99}
Every definable group that is not definably compact contains a $1$-dimensional torsion-free definable subgroup.
\end{fact}

Recall that every group $G$ definable in an o-minimal structure contains a \textbf{maximal normal definable torsion-free subgroup} that we denote by $\bf{\mtf(G)}$. When $G$ is abelian, then $G = \mtf(G) + S$, where $S$ is the unique $0$-Sylow subgroup of $G$.

\begin{prop}\label{prop:subgr}
Let $\Rs = (R, +, 0, \cdot)$ be a definable ring. Set $G = (R, +)$. The following definable subgroups of $G$ are ideals in $\Rs$:

\begin{enumerate}[(i)]

\item The $n$-torsion subgroup $T_n$, for each $n \in \N$;  

\item the definably connected component $R^0$ of $0$;

\item the maximal torsion-free definable subgroup $\mtf(G)$;

\item the $0$-Sylow subgroup $S$.
\end{enumerate}

Moreover, $T_n \se \Ann(R^0)$.  If $R$ is definably connected,  
\[
\Ann(R) = (\Ann(\mtf(G)) \cap \mtf(G)) + S.
\]
\medskip
 It follows that $\Ann(R)$ is definably connected, the additive subgroup of $R/\Ann(R)$ is torsion-free, and $R^2 \se \mtf(G)$.

 \end{prop}

\begin{proof}
\begin{enumerate}[(i)]

\item Suppose $x$ is a torsion element of $(R, +)$ and let $n$ be a positive integer such that $nx = 0$.  By distribution, 
$n(xa) = (nx) a = 0 = a (nx) = n(ax)$ for each $a \in R$. Therefore, both $xR$ and $Rx$ are contained in the $n$-torsion subgroup of $(R, +)$, which is finite by \Cref{fact:finite-torsion}. As $xR^0$ and $R^0x$ are definably connected, they must be trivial, and $x \in \Ann(R^0)$.

\medskip
\item Because $R^0$ is divisible, there is a finite subgroup $F$ of $G$ such that $G = F \oplus R^0$. As $F$ is contained in the torsion subgroup of $G$, $F \subseteq \Ann(R^0)$ by (i), and $R^0$ is an ideal.

\medskip
\item For each $x \in R$, the maps $G \to G$, given by $g \mapsto xg$ and $g \mapsto gx$ are definable homomorphisms. Since the image of a torsion-free definable group by a definable homomorphism is torsion-free, $xA + Ax \subseteq A$ for each $x \in R$, and $A$ is an ideal.

\medskip
\item The image of a $0$-group by a definable homomorphism is a $0$-group. It follows that $xS + Sx \subseteq S$ for each $x \in R$, and $S$ is an ideal.

\end{enumerate}

\medskip
Suppose $R$ is definably connected. Set $A = \mtf(R)$. Then $R = A + S$. Because $S$ is the definable subgroup generated by the torsion of $G^0$, $S \subseteq \Ann(R)$ by (i).

Let $r \in \Ann(R)$, $r = a + s$, with $a \in A$ and $s \in S$. Since $s \in \Ann(R)$, $a \in \Ann(R)$. In particular, $a \in \Ann(A)$ and $\Ann(R) \subseteq \Ann_A(A) + S$. The other inclusion follows from $S \subseteq \Ann(R)$ and $R = A + S$.

If $r_i = a_i + s_i$ for $i = 1, 2$, then $r_1r_2 = a_1a_2$, because $S \se \Ann(R)$, so $R^2 \se A$, as required.
\end{proof}

\begin{cor}\label{cor:null-tf}
Let $R$ be a definably connected ring. Then $R$ is a null ring if and only if $\mtf(R, +)$  is a null ring.
\end{cor}
%------------------------
%\medskip
There are (even commutative unital) ring containing  elements that are not units neither zero-divisors. Every ring that is not isomorphic to its total ring of fractions is an example. But in 
definable rings this cannot happen.

\begin{lem}\label{lem:zero-unit}
Let \rng\ be a definable  ring. Every non-zero $a \in R$ is either a unit or a zero-divisor. In particular, if $\Rs$ is not unital, every element is a zero-divisor.
\end{lem}

\begin{proof}
Suppose $R$ is definably connected and let $a \in R$, $a \neq 0$. We want to show that if $a$ is not a zero-divisor, then $R$ is unital and $a$ is a unit.

Because $a$ is not a left zero divisor, $\Ann_2(a) = \{0\}$ and $\dim aR = \dim R$. Therefore $aR = R$, as $R$ is definably connected. In particular, $ax = a$ for some $x \in R$. For each $y \in R$,
$axy = ay$, so $a(xy - y) = 0$. Since $a$ is not a zero-divisor, $x$ is a left identity.

Because $a$ is not a right zero divisor, $\Ann_1(a) = \{0\}$, $\dim Ra = \dim R$ and $Ra = R$. In particular, $a = ya$ for some $y \in R$. For each $z \in R$, 
$za = zya$, so $(z - zy)a = a$. Since $a$ is not a zero-divisor, $y$ is a right identity.

We have $xy = y = x = 1$. Moreover, from $aR = R = Ra$ we know there are $u, v \in R$ such that $au = 1 = va$.
Thus $vau = u = v$ and $a$ is a unit.

Suppose $R$ is not definably connected. There is a finite additive subgroup $F$ such that 
$R = F \oplus R^0$. Write $a = x + y$, with $x \in F$, $y \in R^0$.

By the connected case above, $y$ is either a unit or a zero-divisor in $R^0$. If $y$ is a zero-divisor, then $a$ is a zero-divisor too, because $F \se \Ann(R^0)$ by \Cref{prop:subgr}. 

If $y$ is a unit, then $R^0$ is unital. By \Cref{prop:subgr}, $(R^0, +)$ is torsion-free, $F$ is an ideal and $R = F \times R^0$. It is well-known that in a finite ring every element is either a unit or a zero-divisor. If $x$ is a unit in $F$, then $a$ is a unit too. If $x$ is a zero-divisor in $F$, then so is $a$ in $R$.
\end{proof}
 
\begin{prop}\label{prop:1-dim}
 Every $1$-dimensional definably connected ring is either a null ring or a real closed field.
\end{prop}

\begin{proof}
Let $R$ be a $1$-dimensional definably connected ring. Because $R$ is definably connected, the smallest definable subring $R(x)$ containing $x$ is infinite for each $x \neq 0$. So $R(x) = R$ and $R$ is commutative.

Let $a \neq 0$. If $a^2 = 0$, then $a \in \Ann(a)$. It follows that $\Ann(a) = R$ and $a \in \Ann(R)$. Therefore, $\Ann(R) = R$ and $R$ is a null ring.

If $a^2 \neq 0$, $a \notin \Ann(a)$ and $\Ann(a) = \{0\}$. Therefore $aR = Ra = R$, $R$ is unital and $a$
is a unit. Similarly for any other $x \in R$, so $R$ is a field. 

By Pillay, every $1$-dimensional definable field in an o-minimal structure is real closed.
\end{proof}

We do not know whether every torsion-free abelian definable group is a direct sum of $1$-dimensional definable subgroups. When this is the case, we can infer similar properties on definable subgroups and complements.

\begin{lem}\label{lem:tf-splitting}
Let $(G, +)$ be an abelian torsion-free $n$-dimensional definable group. Suppose there are definable $1$-dimensional subgroups $A_1, \cdots, A_n$ such that
\[
G = A_1 \oplus \cdots \oplus A_n.
\]

\medskip
Then for each definable subgroup $H$ of $G$:

\begin{enumerate}[(i)]
\item there are $1$-dimensional definable subgroups $H_i$ such that
\[
H = H_1 \oplus \cdots \oplus H_s
\]

\item there are $1$-dimensional definable subgroups $B_i$ such that
\[
G = H \oplus B_1 \oplus \cdots \oplus B_k.
\]

\end{enumerate}
\end{lem}

\begin{proof}
We prove our claims by induction on $m = \dim H$, starting with $(ii)$. Suppose $m = 1$. We will show our claim by induction on $n = \dim G$. We can assume $n > 1$. The intersection
\[
H \cap (A_1 \oplus \cdots \oplus A_{n-1} )
\]

\medskip
is a definable subgroup of the $1$-dimensional torsion-free $H$, so it is either trivial or equal to $H$. In the first case, $G = H \oplus A_1 \oplus \cdots \oplus A_{n-1}$, and we are done. In the second case, 
$H \subset A_1 \oplus \cdots \oplus A_{n-1}$. By induction hypothesis,
\[
A_1 \oplus \cdots \oplus A_{n-1} = H \oplus B_1 \oplus \cdots \oplus B_{n-2}
\] 
so
\[
G = H \oplus B_1 \oplus \cdots \oplus B_{n-2} \oplus A_n,
\]
as we wanted. Suppose now $\dim H > 1$. Let $J$ be a $1$-dimensional definable subgroup of $H$ from \Cref{fact:1dimtf}. Set $\ol G = G/J$ and $\ol H = H/J$. By induction hypothesis, $\ol G = \ol H \oplus \ol K$
for some definable subgroup $\ol K$ that is a direct sum of $1$-dimensional definable subgroups. Let $K$ be the pre-image of $\ol K$ in $G$. By the previous $1$-dimensional case, 
$K = J \oplus B_1 \oplus \cdots \oplus B_k$, so
$G = H \oplus B_1 \oplus \cdots \oplus B_k$, as we wanted.

To show $(i)$ holds, 
%assume first $\dim H = 2$ and let $A$ be a $1$-dimensional definable subgroup of $H$. By $(ii)$ above, 
%\[
%G = A \oplus B_1 \oplus \cdots \oplus B_{n-1}.
%\]
%The intersection $B = H \cap (B_1 \oplus \cdots \oplus B_{n-1})$ cannot be trivial, otherwise $\dim G = n+1$, nor equal to $H$, because $A$ is not included, so it has to be $1$-dimensional. Therefore, $H = A \oplus B$, as we wanted. 
suppose $\dim H > 1$ and let $A$ be a definable subgroup of $H$ of codimension $1$. By $(ii)$, 
$G = A \oplus B_m \oplus \cdots \oplus B_n$. The intersection 
$B = H \cap (B_m \oplus \cdots \oplus B_n)$ must be $1$-dimensional because of dimensions, 
$H = A \oplus B$ because $A \cap B$ is trivial, and $A$ is a direct sum of $1$-dimensional subgroups by induction hypothesis.
\end{proof}

Even if $G$ is not necessarily a direct sum of $1$-dimensional definable subgroups, \Cref{lem:tf-splitting} holds abstractly:

\begin{lem}\label{lem:tf-splitting-abstractly}
Let $(G, +)$ be an abelian torsion-free definable group. For each definable subgroup $H$ there are subgroups $A_1, \dots, A_k$ abstractly isomorphic to $1$-dimensional definable groups such that
\[
G = H \oplus A_1 \oplus \cdots \oplus A_k
\] 
and for each $i = 1, \dots, k-1$ the subgroup $H \oplus A_1 \oplus \cdots \oplus A_i$ is definable.
\end{lem}

\begin{proof}
By \Cref{fact:1dimtf}, \Cref{fact:tf-divisible} and \Cref{fact:divisible-split}. 
%(\textbf{Alternatively, Prop: every abelian torsion-free definable group is isomorphic to a group that definably split and then Corollary \Cref{lem:tf-splitting-abstractly} using \Cref{lem:tf-splitting}}).
\end{proof}

%
%--------------------------

\noindent
Recall that a $K$-vector space $V$ is an \textbf{associative $K$-algebra} when there is a multiplication $\cdot$ on $V$ making $(V, +, \cdot)$ a ring and $\cdot$ is compatible with the scalar multiplication $K \times V \to V$, $(k, v) \mapsto kv$. That is, for any $r, s \in K$ and $u, v \in V$,

\[
(ru) \cdot (sv) = (rs)(u \cdot v).
\]

 \medskip
 The most important example of associative $K$-algebra is the ring of $n \times n$ matrices $M_n(K)$ with usual scalar multiplication. In fact, every finite-dimensional associative $K$-algebra is isomorphic to a subalgebra of $M_n(K)$ for some $n \in \N$.

A $K$-vector space $V$ is \emph{definable} when $K$ is a definable field, $V$ is a definable abelian group and the scalar multiplication map $K \times V \to V$, $(k, v) \mapsto kv$ is definable. Similarly, we say that $A$ is a \textbf{definable associative $K$-algebra} when $A$ is an associative $K$-algebra that is a definable ring and a definable $K$-vector space.  

%It is actually equivalent to 
% 
% \[
%k(uv) = (ku) v = u (kv)  
% \]
% 
% for each $k \in K$, $u, v \in V$.

\begin{prop} \label{prop:algebras}
Let $\Rs = (R, +, 0, \cdot)$ be a definable ring. If $R$ is a definable $K$-vector space for some definable real closed field $K$, then $R$ is a definable associative $K$-algebra.
\end{prop}

\begin{proof}
Let $1 \in K$ be the multiplicative identity. As $R$ is a $K$-vector space, $1x = x$ for each $x \in R$. For $R$ to be an associative $K$-algebra, we need
\[
(ru) \cdot (sv) = (rs)(u\cdot v)
\]

\medskip
for any $r, s \in K$ and $u, v \in R$. Fix $u, v \in R$, $s \in K$ and set
\[
G = \{r \in K: (ru)\cdot(sv) = (rs)(u\cdot v)\}.
\] 

\medskip
It is easy to check that $G$ is a definable subgroup of $(K, +)$. Define now
\[
H = \{s \in K: u \cdot (sv) = s(u\cdot v)\}.
\]

\medskip
Once again, $H$ is a definable subgroup of $(K, +)$. Since 
$u \cdot (1v) = u \cdot v = 1(u \cdot v)$, $1 \in H$. Because $(K, +)$ is $1$-dimensional torsion-free, $(K, +)$ has no proper non-trivial definable subgroup and $H = K$. Moreover, $1u = u$ and $1s = s$, so $1 \in G$ and $G = K$ too, as wanted. 
\end{proof}

 %%================================  
%%%===============================
  
%\medskip
\section{Nilpotent rings}

Let $R$ be an arbitrary ring. Recall that $a \in R$ is called a \textbf{nilpotent element} if $a^k = 0$ for some $k \in \N$. We will say that $a$ is \textbf{$\bf{k}$-nilpotent} when $a^k = 0$ and $a^{k-1} \neq 0$.

A \textbf{nil ring} is a ring whose elements are all nilpotent. A nil ring $R$ is a \textbf{nilpotent ring} when there is some $k \in \N$ such that $R^k = \{0\}$. That is, the product of any $k$ elements in $R$ is zero. The main example of a nil (and nilpotent) definably connected ring is the ring $T_n(K)$ of strictly upper triangular $n \times n$ matrices with coefficients in a real closed field $K$. We will show that every definably connected nil ring \Rs\ essentially embeds in a product of $T_{n_i}(K_i)$ for suitable \Rs-definable real closed fields $K_i$'s.

Not all nil rings are nilpotent rings, but we will show that all \emph{definable} nil rings are nilpotent. We will also prove that a definable ring is nilpotent if and only if $0$ is the only idempotent element. We start by finding an uniform bound for the nilpotency of the elements.

%\bigskip
\begin{prop}\label{prop:bound}
Let $R$ be a definably connected $n$-dimensional ring. 

If $a \in R$ is a nilpotent element, $a^{n+1} = 0$.
If $R$ is unital, $a^n = 0$.
\end{prop}

\begin{proof}
By induction on $n$. If $n=1$, because $a$ is a zero divisor, the multiplication on $R$ is trivial, and $a^2 = 0$. If $R$
is unital, $R$ is a real closed field (\Cref{prop:1-dim}) and $a = 0$ is its only nilpotent element.

Assume $n > 1$. Suppose $k$ is the smallest integer such that $a^k = 0$. Set $b = a^{k-1}$. Then $a \in \Ann(b)$.
If $\Ann(b) \neq R$, then $\dim \Ann(b) < n$, because $R$ is definably connected. By induction hypothesis, $a^n = 0$
and we are done.  

If $\Ann(b) = R$, $b \in \Ann(R)$ and $\Ann(R)$ is an infinite ideal of $R$. If $a \in \Ann(R)$, then $a^2 = 0$. 
Assume $a \notin \Ann(R)$. Then $\dim \Ann(R)=p < n$. Let $\bar{a}$ be the image of $a$ in the quotient ring $R/\Ann(R)$. By induction hypothesis, $\bar{a}^{p+1} = \bar{0}$. That is, $a^{p+1} \in \Ann(R)$ and $a^{p+2} = 0$. Because $p < n$, $p+2 \leq n + 1$.

Assume $R$ is unital, $\Ann(b) = R$ and $a \notin \Ann(R)$. Then the quotient ring $R/\Ann(R)$ is unital, $\bar{a}^p = \bar{0}$ by induction hypothesis and $a^n = 0$.  
\end{proof}

\begin{lem}\label{lem:generator-matters}
Let $\Rs = (R, +, 0, \cdot)$ be a definable ring such that $(R, +)$ is torsion-free. For any $x, y \in R$,
\[
\la x \df = \la y \df \ \Longrightarrow\ xR = yR\ \mbox{ and } \ Rx = Ry.
\]
\end{lem}

\begin{proof}
We will show that for each $a \in R$, 
\[
x \la a \df = y \la a \df.
\]

This suffices, because if $\dim R = n$ there are $a_1, \dots, a_n \in R$ such that
\[
R = \la a_1 \df + \cdots + \la a_n \df.
\]

Indeed, let $a_1 \in R$, $a_1 \neq 0$ and set $A_1 = \la a_1 \df$. If 
$R = A_1$, any $a_2, \dots a_n$ will meet our requirement. Otherwise, let $a_2 \notin A_1$ and set $A_2 = \la a_2 \df$. Then $\dim (A_1 + A_2) \geq 2$. We iterate until we reach $n = \dim R$. 

Set $H = \la x \df = \la y \df$, fix $a \in R$ and set $A = \la a \df$.  If $xa = 0$, then $x \in \Ann_1(a)$ and $a \in \Ann_2(x)$. Therefore
$H \se \Ann_1(A)$ and $x A = y A = \{0\}$. Equivalently, $\Ann_2(x) = \Ann_2(y)$.

Suppose $xa = b \neq 0$. Set  $B = \la b \df$ and $G = \{r \in R: ra \in B\}$. 
As $G$ is a definable subgroup of $(R, +)$ containing $x$, $H \subseteq G$. Thus $Ha \se B$.
In particular, $ya \in B$.

Set $A' = \{r \in R: yr \in B\}$. Since $A'$ is a definable subgroup containing $a$, then  
$A \se A'$. In particular,  $yA \se B$.

On the other hand, $b \in xA \se B$, so $B = xA$. It follows that $yA \se xA$. Since $\Ann_2(x) = \Ann_2(y)$, $\dim yA = \dim xA$ and $yA = xA$, as required. The other equality is similar.
\end{proof}

 \begin{prop} \label{prop:nil-flag}
Let $\Rs = (R, +, 0, \cdot)$ be a $n$-dimensional definable nilpotent ring such that $(R, +)$ is torsion-free. There are definable ideals
\[
\{0\} = I_0 \subset I_1 \subset \dots \subset I_n = R
\]

\medskip
such that $\dim I_k = k$ and $I_{k+1}/I_k \subseteq \Ann(R/I_k)$ for all $k \in \{0, \dots, n-1\}$, and $\Ann(R) = I_k$ for some $k \in \{1, \dots, n\}$. 
 \end{prop}

\begin{proof}
By induction on $n$. If $n = 1$, by \Cref{prop:1-dim}, $R = I_1 = \Ann(R)$, as required. Suppose $n > 1$.

We first show that $\Ann(R)$ is infinite. Since every torsion-free definable group is definably connected, we only have to show that $\Ann(R)$ is non-trivial.

Suppose, by way of contradiction, $\Ann(R) = \{0\}$ and let $k$ be the smallest positive integer such that $R^{k+1} = \{0\}$. 
Let $x \in R^k$, $x \neq 0$. Since $\Ann(R) = \{0\}$, there is $y \in R$ such that $xy \neq 0$ or $yx \neq 0$. Therefore, $R^{k}y \neq \{0\}$ or $yR^k \neq \{0\}$. Either way, $R^{k+1} \neq \{0\}$, contradiction. So $\Ann(R)$ is infinite, as claimed. If $\dim \Ann(R) = k$, set $I_k = \Ann(R)$.

If $k > 1$, take $I_1$ to be a $1$-dimensional definable subgroup of $\Ann(R)$ (\Cref{fact:1-dim}) and 
$\overline{R} = R/I_1$. By induction hypothesis, there are definable ideals
\[
\{0\} = J_0 \subset J_1 \subset \dots \subset J_{n-1} = \overline{R}
\]

\medskip
such that $\dim J_k = k$ and $J_{k+1}/J_k \subseteq \Ann(\overline{R}/J_k)$ for all $k = 0, \dots, n-2$. For $i = 1, \dots, n-1$, set $I_{i+1}$ to be the pre-image of $J_i$ in $R$.
\end{proof}

\begin{cor} \label{cor:nilpotent-Rn+1=0}
If $R$ is a $n$-dimensional definably connected nilpotent ring, then $R^{n+1} = \{0\}$.
\end{cor}

\begin{proof}
If $(R, +)$ is torsion-free, the claim follows from \Cref{prop:nil-flag}. Otherwise, we can reduce to the torsion-free case by \Cref{prop:subgr}, because $R = A + S$, where $A$ is torsion-free and $S \subseteq \Ann(R)$.
\end{proof}

%---------------
%----------------
\begin{theo}\label{theo:nilpotent}
Let $\Rs = (R, +, 0, \cdot)$ be a definably connected nilpotent ring.  %Let $s$ be the maximal dimension of a monogenic subring of $R/\Ann(R)$. 
 If $\Rs$ is not a null ring, there is a positive integer $s$ and \Rs-definable real closed fields $K_1, \dots, K_s$  such that 

 \begin{itemize}

\medskip
\item $K_i$ is not definably isomorphic to $K_j$ if $i \neq j$;

\medskip
\item $R$ is a direct product of rings
\[
R = R_0 \times R_1 \times \cdots \times R_s
\]

\medskip \noindent
where $R_0 \subset \Ann(R)$ (and so $R^2 \subset R_1 \times \cdots \times R_s$), and for $i=1, \dots, s$, $R_i$ is a finite dimensional associative $K_i$-algebra that is not a null ring; 

\medskip
\item if $\mtf(R, +)$ is a direct sum of $1$-dimensional definable subgroups, then 
for $i >0$, $R_i$ is a definable associative $K_i$-algebra and $s$ is the maximal dimension of a monogenic definable subgroup of $R/\Ann(R)$.

 \end{itemize}
 \end{theo}

\begin{proof}
 By induction on $\dim R$. If $\dim R = 1$, $R$ is a null ring and there is nothing to prove.  Set $\dim R = n+1$, $n \geq 1$. Set $G = (R, +)$. We can assume $G$ is torsion-free. 

Indeed, if $G$ is not torsion-free, let $S$ be its infinite $0$-Sylow so that 
$R = \mtf(G) + S$. After  proving our claims for the torsion-free definable ideal $\mtf(G) = R'_0 \times R'_1 \times \cdots \times R'_s$, we can set $R_0 = R'_0 + S$ and $R_i = R'_i$ for $i \geq 1$, since $S \se \Ann(R)$ and $R^2 \se \mtf(G)$ (\Cref{prop:subgr}). So let $(R, +)$ be torsion-free and fix a chain of ideals $I_j$ as in \Cref{prop:nil-flag}. Recall that every definable torsion-free group is definably connected, from which it follows that whenever $H \se K$ are definable additive subgroups, $H \neq K$ implies $\dim H < \dim K$.  

\bigskip
Assume $\dim \Ann(R) = n$. That is, \fbox{$\Ann(R) = I_n$}. We will show that in this case $s = 1$ and $R = R_0 \times R_1$, where $R_0 \se \Ann(R)$, and $R_1$ is a 2-dimensional commutative associative algebra over a $\Rs$-definable real closed field $K$.

Recall that $R/I_n$ has trivial multiplication, so $xy \in I_n$
for each $x, y \in I_n$. We claim that $R$ is commutative. 

If not, let $a \notin Z(R)$. Therefore, $C(a)$ is a proper definable subring of $R$ containing $a$ and $\dim C(a) \leq n-1$. On the other hand, $I_n = \Ann(R) \subseteq C(a)$, so $\dim C(a) = n-1$ and $C(a) = I_n$.
Because $a \in C(a)$, we have  
$a \in I_n \subseteq Z(R)$, contradiction.  

Note that $\Ann_1(R) = \Ann_2(R) = I_n$, otherwise $R^2 = 0$. Therefore, $xR$ has positive dimension for each $x \notin I_n$.

Fix $a \notin I_n$. Since $I_n \subseteq \Ann(a)$, it must be $\dim aR = 1$. 
If $x \notin I_n$, then $x \notin \Ann(a)$, so $xa \neq 0$. As for $aR$, $\dim xR = 1$. 
Because $xa = ax \in xR \cap aR$, it follows that $xR = aR$. That is, $R^2 = aR$. 

Set $\overline{R} = R/I_n$. As $I_n = \Ann(a)$, the \Rs-definable map
\[
f_a \colon \overline{R} \to aR
\] 

%\medskip
given by
\[
x + I_n \mapsto ax
\]

\medskip
is a well-defined group isomorphism. 

For each $x + I$, $y + I \in \overline{R}$, define
\[
(x + I) \otimes (y + I) = f_a\inv (xy).
\] 

We claim that $(\overline{R}, +, \otimes)$ is a real closed field.  

%For ease of notation, set 
%\[
%g(x) = f_a\inv (x).
%\]

%\medskip
To show that $\otimes$ is associative, we need to ensure that for each $x, y, z \in R$
 \[
f_a\inv (f_a\inv (xy)z)  = f_a\inv (x f_a\inv (yz)),
\]

\medskip
which is equivalent to 
\[
f_a\inv (xy)z  = x f_a\inv (yz),
\]

because $f$ is a bijection. Fix $y, z \in R$ and set 
\[
G = \{x+ I_n \in \overline{R} : f_a\inv (xy)z  = x f_a\inv (yz) \}. 
\]

\medskip
Note that $G$ is well-defined because $xy$ only depends on the classes $x+I_n$ and $y+I_n$. It is easy to check that $G$
is a definable subgroup of $(\overline{R}, +)$. Whenever $y$ and $z$ do not belong to $I_n$, by setting $x = a$, we get
\[
f_a\inv (ay)z  = yz = a f_a\inv (yz).
\]

\smallskip
That is, $a+ I \in G$. Since $(\overline{R}, +)$ is $1$-dimensional torsion-free, it has no proper non-trivial definable subgroup, so $G = \overline{R}$, and $\otimes$
is associative, as claimed. Distribution of multiplication with respect to addition follows from the fact that $f_a\inv$ is an additive homomorphism.

%\medskip
Note that for each $x \in R$,
\[
(a + I) \otimes (x + I) = (x + I) \otimes (a + I) = f_a\inv (ax) = (x + I).
\]

\medskip
Thus $(a + I)$ is a multiplicative identity. For each $(x + I) \in \overline{R}$, 
\[
(x + I)^{-1} = f_x^{-1}(a^2). 
\]

Therefore $K = (\overline{R}, +, \otimes)$ is a field. As $K$ is $1$-dimensional, it is real closed by \cite{Pi88}, as claimed. Also note that $K$ is \Rs-definable as $\ol R = R/\Ann(R)$ and $f_a$ are \Rs-definable. Now for $x, y \in aR$ we can set
\[
x \otimes y = f_a(f_a\inv (x)  \otimes f_a\inv (y))  
\] 
 
\medskip \noindent
obtaining a definably isomorphic real closed field $(aR, +,  \otimes)$ with multiplicative identity $a^2$. 
 
%\medskip
Let $B$ be a subgroup of $(R, +)$ from \Cref{fact:tf-divisible} and \Cref{fact:divisible-split} such that
\[
R = I_n \oplus B.
\]
$(B, +)$ is isomorphic to $(\overline{R}, +)$ through the canonical homomorphism $R \to R/I_n$, so $B$ too admits a multiplication making it a field isomorphic to $K$.
 
 Set $R_1 = aR \oplus B$. Clearly, $R_1$ is a $2$-dimensional $K$-vector space.  To conclude that $R_1$ is an associative algebra, we need
\[
(ru)(sv) = (rs)(uv)
\]

\medskip
for any $r, s \in K$ and $u, v \in R_1$. Write $u = x_1 + y_1$, $v = x_2 + y_2$ with $x_i \in aR$ and $y_i \in B$. Since $aR \subseteq \Ann(R)$, we have
\[
(ru)(sv) = (rx_1 + ry_1)(sx_2 + sy_2) = ry_1sy_2
\]
and
\[
(rs)(uv) = (rs)y_1y_2.
\]

\medskip
As $B$ is a $K$-vector space, equality follows. 

If $(R, +)$ is a direct sum of $1$-dimensional definable subgroups, $B$ can be taken to be definable by \Cref{lem:tf-splitting}, so $R_1 = aR \oplus B$ is a definable $K$-vector space and by \Cref{prop:algebras}, $R_1$ is a definable associative $K$-algebra. 

Finally, as $aR$ is a divisible subgroup of the abelian group $I_n$, there is a (possibly not definable) subgroup $H$ such that $I_n = aR \oplus H$. Since $H \subset \Ann(R)$, clearly $H$ is an ideal, and
\[
R = H \times (aR \oplus B).
\]

\medskip
 Therefore, $R_0 = H$ and $R_1 = aR \oplus B$ satisfy our claims. 

%------------------------
 %------------------------
  \bigskip
 Assume now \fbox{$\Ann(R) \neq I_n\ \&\ I_n^2 = \{0\}$}. That is, $I_n$ has trivial multiplication. We will show that also in this case $s = 1$ and $R = R_0 \times R_1$,
where $R_0 \subset \Ann(R)$, $R^2 \subset R_1$ and $R_1$ is an associative algebra over a $\Rs$-definable real closed field $K$. The difference with the previous case is that $R_1$ can have arbitrary dimension (see \Cref{ex:In-trivial}).
%Note that because we already proved this is the case when $\Ann(R) = I_n$, we do not need to assume otherwise when applying the induction hypothesis.

Suppose $k = \dim I_n - \dim \Ann(R)$. By \Cref{lem:tf-splitting-abstractly}, there are additive subgroups 
$A_1, \dots, A_k, B$, each isomorphic to a definable $1$-dimensional group, such that
\begin{equation}
I_n = \Ann(R) \oplus A_1 \cdots \oplus A_k \quad \mbox{and} \quad R= I_n \oplus B.
\end{equation}

\medskip \noindent
Because $I_n$ has trivial multiplication, for each nonzero $x \in \oplus_{i = 1}^k A_i$, $\dim \Ann(x) = n$, since $x \notin \Ann(R)$, and at least one of $xR$ or $Rx$ is infinite, hence $1$-dimensional. 
In particular, for each $y \notin I_n$, $xy \neq 0$ or $yx \neq 0$, because $R = I_n + \la y \df$, and if $xy = yx = 0$, then $\la y \df \se \Ann(x)$ and $x \in \Ann(R)$, contradiction.

Let  $A = A_1$, fix a nonzero $a \in A$ and, without loss of generality, assume $\dim aR = 1$.   

Set $\overline{R} = R/I_n$. Since $\Ann(a) = I_n$, as for the case $\Ann(R) = I_n$ above, the definable map
\[
f_a \colon \overline{R} \to aR
\] 

%\medskip
given by
\[
x + I_n \mapsto ax
\]

\medskip
is a well-defined group isomorphism. However, unlike the case $\Ann(R) = I_n$, $R$ is not necessarily commutative and $aR$ may not coincide with $R^2$. In particular, for elements  $x+I_n, y+I_n \in \overline{R}$, the product $xy$ may not be contained in $aR$, the multiplication on $\overline{R}$ from the previous case is not well-defined, and we need a different strategy to define a product on $\ol R$ and $aR$.

%\bigskip  
Assume first $(R, +)$ is a direct sum of $1$-dimensional definable subgroups. Thus each $A_i$ and $B$ can be taken to be definable by \Cref{lem:tf-splitting}.  

Fix  a nonzero $b \in B$. As noted above, $ab \neq 0$, otherwise $B \subset \Ann(a)$ and $aR = \{0\}$, contradiction.
Moreover, $\alpha b \neq 0$ for each $\alpha \in A$, otherwise $A = \la \alpha \df \subset \Ann_1(b)$, in contradiction with $ab \neq 0$. Therefore, $aB = Ab = aR$. Set $C = aR$. There are definable group isomorphisms $f \colon A \to B$ and $g \colon C \to A$ given by
\[
f(x) = u\ \Leftrightarrow\ xb = au \qquad \mbox{ and } \qquad g(x) = v \Leftrightarrow\ x = vb. 
\]
 
For $x, y \in A$, define
\[
x \otimes y = g(xf(y)).
\]

\medskip
This operation is well-defined because $\alpha \beta \in C$ for each $\alpha \in A$ and $\beta \in B$. Indeed, for $\alpha \neq 0$, $\alpha b \neq 0$, and $\alpha b \in \alpha R \cap aR$, so $\alpha R = aR = \alpha B$.  

\medskip
To show that $\otimes$ is associative, we need to ensure that for each $x, y, z \in A$
 \[
g(g(xf(y))f(z) = g(xf(g(y f(z)))), 
\]

\medskip
which is equivalent to 
\[
 g(xf(y))f(z) = xf(g(y f(z))),
\]

\medskip
because $g$ is a bijection. Fix $y, z \in A$ and set 
\[
G = \{x \in A : g(xf(y))f(z) = xf(g(y f(z))) \}. 
\]

\medskip
 Since $g$ is a group homomorphism, $G$ is a definable subgroup of $(A, +)$. 
 To see that $a \in G$, note that for each $x \in A$, $af(x) = xb$ and $g(xb) = x$. Moreover, for each $c \in C$, 
 $g(c)b = c$. Thus
 \[
 g(af(y))f(z) = g(yb)f(z) = yf(z)            
 \]
and
\[
af(g(y f(z))) = g(y f(z)) b = yf(z). 
\]
 
 \medskip
Since $(A, +)$ has no proper non-trivial definable subgroup, $G = A$, and $(A, +, \otimes)$ is a ring. For each $x \in A$
\[
a \otimes x = g(af(x)) = g(xb) = x \quad \mbox{ and } \quad x \otimes a = g(xf(a)) = g(xb) = x. 
\]

\medskip
That is, $a$ is an identity for $\otimes$ and $(A, +, \otimes) = K$ is a real closed field by \Cref{prop:1-dim}. Using the definable isomorphisms $f \colon A \to B$ and $g \colon C \to A$ one can define corresponding multiplications on $B$ and $C$ obtaining two real closed fields definably isomorphic to $K$. 

The same process can be repeated by replacing $A$ with $A_i$ for $i = 2, \dots, k$ to show that each admits a multiplication making it a real closed field definably isomorphic to $K$. For each $i = 1, \dots, k$,
fix $a_i \in A_i$, $a_i \neq 0$. For ease of notation, set $B = A_{k+1}$ and $b = a_{k+1}$.

Since $A_i$ is $1$-dimensional torsion-free, $a_i$ is a definable generator for each $i$. It follows from the decompositions (1) that
 \[
\sum_{x \in R} xR + \sum_{x \in R} Rx = \sum_{i = 1}^{k+1} a_iR + \sum_{i = 1}^{k+1} Ra_i.
\]

Define
\[
J = \left (\sum_{i = 1}^{k+1} a_iR + \sum_{i = 1}^{k+1} Ra_i \right) \cap \Ann(R).
\]

\medskip
Because $J$ is a definable subgroup of $(R, +)$ contained in $\Ann(R)$, $J$ is an ideal. If $J$ is properly contained in $\Ann(R)$, set $R_1 = J \oplus \sum_{i = 1}^{k+1}A_i$ and $R_0$ a complement of $J$ in $\Ann(R)$ from \Cref{fact:divisible-split} , so that $R = R_0 \oplus R_1 = R_0 \times R_1$. By induction hypothesis, $R_1$ is a definable associative $K$-algebra, as required. Note that $R^2 \se R_1$ by construction.

If $J = \Ann(R)$, we claim that $R$ is a definable associative $K$-algebra. For $i = 1, \dots, k$ we already know each of $a_iR$ or $Ra_i$
is either trivial or a definable $1$-dimensional $K$-vector space. For $b \in B$, $bR$ and $Rb$ may have dimension greater than
$1$. Indeed, they may have arbitrarily large dimension. However, 
\[
bR = bA_1 \oplus \cdots \oplus bA_{k+1}
\]

\medskip
where we know each $A_i$ is a $1$-dimensional definable $K$-vector space so, whenever non-trivial, $bA_i$ is too. Similarly for $Rb$. Therefore, $R$ is a definable $K$-vector space. By \Cref{prop:algebras}, it is a definable associative $K$-algebra. 

Note that for each $x \in R/\Ann(R) = R_1/\Ann(R_1)$, $\la x \df = Kx$, so $s = 1$ is the maximal dimension of a monogenic definable subgroup of $R/\Ann(R)$, as claimed.

Consider now the case where $(R, +)$ is not necessarily a direct sum of definable $1$-dimensional subgroups. Let $\Ann(R) = I_s$. Replace $A_1$ with the definable $I_{s+1}/I_s$ and $B$ with the definable $R/I_n$. As already noted, for each   
$b \notin I_n$ and $a \in I_{s+1} \backslash I_s$, at least one of $ab$ or $ba$ is different from $0$, because $I_n = \Ann(a)$. Assume $ab \neq 0$. There are still corresponding definable group isomorphisms $f \colon I_{s+1}/I_s \to R/I_n$
and $g \colon aR \to I_{s+1}/I_s$ allowing to define multiplications making each a definable real closed field $K$. More generally, for $i =1, \dots, k$, $A_i$ can be replaced by the definable $I_{s+i}/I_{s+i-1}$ and $a_i$ can be taken to be any element $I_{s+i}$
that is not in $I_{s+i-1}$.

We still have
\[
\sum_{x \in R} xR + \sum_{x \in R} Rx = \sum_{i = 1}^{k+1} a_iR + \sum_{i = 1}^{k+1} Ra_i,
\]

so that
\[
J = \left (\sum_{i = 1}^{k+1} a_iR + \sum_{i = 1}^{k+1} Ra_i \right) \cap \Ann(R)
\]

\medskip
is a definable ideal and $R_1 = J \oplus \sum_{i = 1}^{k+1}A_i$ is a (possibly not definable) associative $K$-algebra such that $R^2 \subset R_1$.  
If $J$ is properly contained in $\Ann(R)$, let $R_0$ be a complement of $J$ in $\Ann(R)$ from \Cref{fact:divisible-split} , so that $R = R_0 \oplus R_1 = R_0 \times R_1$. 

Note that $K$ is \Rs-definable because, as noted above, $I_n = \Ann(a)$ for each $a \in I_n$, $a \notin \Ann(R)$, so $R/I_n \cong aR$ is $\Rs$-definable.

   \bigskip
 Assume now \fbox{$I_n^2 \neq \{0\}$}. That is, $I_n$ is not a null ring. Suppose first $(R, +)$ splits into a direct sum of definable $1$-dimensional subgroups, and recall we are working under the assumption that $(R, +)$ is torsion-free.
 
 Let $A$ be a 
$1$-dimensional definable subgroup such that $R = I_n \oplus A$ from \Cref{lem:tf-splitting}.
 
% \medskip
We will first prove by induction on $n+1 = \dim R$ that if $R/I_1$ is a null ring, then $s = 1$, as it was shown above for the case $I_n^2 = \{0\}$, and 
\[
R = R_0 \times R_1,
\] 

where $R_0 \se \Ann(R)$, $R$ is a definable associative $K$-algebra for some definable real closed field $K$, and $R^2 \subset R_1$.  

If $n = 1$, then $I_1 = I_n = \Ann(R)$ and our claim was proved above. Suppose $n > 1$.

Since $R/I_1$ is a null ring, $I_n/I_1$ is a null ring too. By induction hypothesis, $I_n = S_0 \times S_1$, where $S_1$ is a definable associative $K$-algebra, for some definable real closed field $K$, and $I_n^2 = I_1 \subset S_1$. In particular, $I_1$ is a $1$-dimensional definable $K$-vector space and can be equipped with a definable multiplication $\otimes$ such that $(I_1, +, \otimes)$ is a definable real closed field definably isomorphic to $K$.

Let $a \in A$ be any nonzero element. Then $a \notin \Ann(R) \subset I_n$ and there is some  $b \in R$ such that $ab$ or $ba$ is nonzero. Without loss of generality, assume $ab \neq 0$. 

Since $A$ is $1$-dimensional torsion-free, $A = \la \alpha \df$ for each nonzero $\alpha \in A$, and $A \cap \Ann_1(b) = \{0\}$. It follows that  $Ab = I_1$ and the map $f \colon A \to I_1$ given by $\alpha \mapsto \alpha b$ is a definable group isomorphism. Using the definable multiplication $\otimes$ on $I_1$ we can then define a multiplication $\otimes$ on $A$ by
\[
a_1 \otimes a_2 = f\inv (f(a_1) \otimes f(a_2)),
\]

\noindent
for any $a_1, a_2 \in A$, making it a real closed field definably isomorphic to $K$.

 By \Cref{lem:tf-splitting}, we can decompose $S_0$ into a direct sum of definable $1$-dimensional subgroups
  \[
  S_0 = A_1 \oplus \cdots \oplus A_k \oplus B_1 \oplus \cdots \oplus B_p
  \]
  
  \medskip 
 such that $B_i \se \Ann(A)$ and $A_i \cap \Ann(A) = \{0\}$. By replacing $b$ with $a$, we can show as done above for $A$ that each $A_i$ is a $1$-dimensional definable $K$-vector space, so $S_1 \oplus A_1 \oplus \cdots \oplus A_k \oplus A$ is a definable $K$-vector space, and an associative $K$-algebra by \Cref{prop:algebras}.
Therefore $R_0 = B_1 \oplus \cdots \oplus B_p$ and $R_1 = S_1 \oplus A_1 \oplus \cdots \oplus A_k \oplus A$ are as required. Clearly, $Kx = \la x \df$ for each $x \in R/\Ann(R)$. This complete the proof that if $R/I_1$ is a null ring, then $R = R_0 \times R_1$ as specifed above.

%\medskip
We now need to prove our theorem when $R/I_1$ is not a null ring. Let $k$ be the smallest index such that $R/I_k$ is a null ring  and set $\ol R = R/I_{k-1}$. 
Then $\ol R$ is not a null ring and we can define $\ol I_1 = I_k/I_{k-1}$ which is a $1$-dimensional definable ideal in $\Ann(\ol R)$ such that  $\ol R /\ol I_1$ is a null ring. 

As proved above, there is a definable real closed field $K$ such that $\ol R = \ol R_0 \times \ol R_1$, where $\ol R_1$
is a definable associative $K$-algebra. By identifying $A$ with $\ol R/\ol I_{n-k+1} = R/I_n$
we can conclude that $A$ is a $1$-dimensional definable $K$-vector space.

%\medskip
By \Cref{lem:tf-splitting} and induction hypothesis, there is a positive integer $r$ and definable real closed fields $K_1, \dots, K_{r}$ such that
\[
I_n = S_0 \times S_1 \times \cdots \times S_{r},
\]

\medskip
where for each $i \geq 1$, $S_j$ is a definable $K_j$-vector space and associative 
$K_j$-algebra, 
$S_0 \subset \Ann(I_n)$, and $I_n^2 \subset S_1 \times \cdots \times S_{r}$.  

%\medskip
As done above, we can decompose $S_0$ into a direct sum of definable $1$-dimensional subgroups
  \[
  S_0 = A_1 \oplus \cdots \oplus A_k \oplus B_1 \oplus \cdots \oplus B_p
  \]
  
 such that $B_i \se \Ann(A)$, $A_i \cap \Ann(A) = \{0\}$. We claim that each $A_i$ is a $1$-dimensional definable $K$-vector space.
 
 Without loss of generality, assume $A_i \cap \Ann_1(A) = \{0\}$, fix a nonzero $a \in A$, and set $C_i = A_i a$. Let $x \in A_i$ and $c \in C_i$ such that $xa = c \neq 0$. 
 Then $xA = C_i$ and the map $f \colon A \to C_i$ given by $f(\alpha) = x \alpha$ is a definable group isomorphism. As done above, using $f$ and the definable product on $A$, $(C_i, +)$ can be equipped with a definable product making it a real closed field definably isomorphic to $K$ and the same can be done for $(A_i, +)$, as claimed.
 
 Suppose $B_i \cap aR$ or $B_i \cap Ra$ is non-trivial. Without loss of generality, assume 
$B_i \cap aR \neq \{0\}$. Let $b \in B_i$ and $r \in R$ such that $ar = b \neq 0$. Then $B_i = Ar$
and $B_i$ is a $1$-dimensional definable $K$-vector space.

After reordering the $B_i$'s, let $B_1, \dots, B_q$ such that $B_i \cap aR = B_i \cap Ra = \{0\}$, and rename $A_{k+1}, \dots, A_p$ the $B_i$'s with non-trivial intersection with $aR$ or $Ra$.

 If $K$ is definably isomorphic to $K_m$, for some 
 $m \in \{1, \dots, r\}$, then $s = r$ and $R_0 = B_1 \oplus \cdots \oplus B_p$, $R_i = S_i$ for $i \neq m$
 and $R_i = S_i \oplus A_1 \oplus \cdots \oplus A_p \oplus A$ for $i = m$.

 If $K$ is not definably isomorphic to any $K_m$, then $s = r+1$, $R_0 = B_1 \oplus \cdots \oplus B_p$, $R_i = S_i$ for each 
 $i \in \{1, \dots, r\}$, $K_s = K$ and $R_s = A_1 \oplus \cdots \oplus A_k \oplus A$.

We need to check that $s$ is the maximal dimension of a monogenic definable subgroup of 
\[
R/\Ann(R) = R_1/\Ann(R_1) \times \cdots \times R_s/\Ann(R_s),
\]

\medskip
which follows from the fact that after writing $x \in R/\Ann(R)$, $x = (x_1, \dots, x_s)$, we have
\[
\la x \df = K_1x_1 \times \cdots \times K_sx_s.
\]

\medskip
Even when $(R, +)$ is not a direct sum of definable $1$-dimensional subgroups, by induction hypothesis $I_n = S_0 \times S'$, where $S'$ is a direct product of associative $K_i$-algebras
and we can replace $A$ with the definable $R/I_n$, $\Ann(A)$ with the definable $\Ann(a)$ for any $a \notin I_n$, to find corresponding associative algebras $R_1, \dots, R_s$. 

Finally, let us understand why $K_1, \dots, K_s$ are $\Rs$-definable. Note that 
\[
\Ann(R) = R_0 \times \Ann(R_1) \times \cdots \times \Ann(R_s)
\]
 and $\Ann(R_i) \neq R_i$ because $R_i$ is not a null ring. Therefore, if $s > 1$, for each $x \in R_1$, 
$x \notin \Ann(R_1)$, $\Ann(x)$ is a proper $\Rs$-definable subring containing $R_2, \dots, R_s$, and $K_2, \dots, K_s$ are $\Rs$-definable by induction hypothesis. Similarly for $K_1$ by taking $x \in R_2$, $x \notin \Ann(R_2)$.
So we can assume $s = 1$ and $R = R_0 \times R_1$, where $R_1$ is a nilpotent associative $K$-algebra.

If $\ol R = R/\Ann(R)$ is not a null ring, then $K$ is $\ol R$-definable by induction hypothesis, and we are done. If $R/\Ann(R)$ is a null ring, then $R^2 \se \Ann(R)$ and $R^3 = \{0\}$. Therefore, $R_1$ is a $3$-nilpotent associative $K$-algebra. We can consider $R_1$ as a subalgebra of some $T_m(K)$, the $K$-algebra of strictly upper triangular $m \times m$ matrices over $K$, and since $I_n^2 \neq \{0\}$, $\dim R_1 \geq 3$. If $\dim R_1 = 3$, then $\dim \Ann(R) = 1$.
In this case, $K$ can be defined over $\Ann(R)$ because $R^2 = \Ann(R)$, so $\Ann(R) = aR$ for each $a \notin \Ann_1(R)$. If $\dim R_1 > 3$, because every matrix in $R_1$ is $3$-nilpotent, it must be made of nilpotent blocks at most $4 \times 4$ and there is some $x \in R_1$ such that $I = \Ann(x)$ has non-trivial multiplication. In this case, $K$ is $I$-definable by induction hypothesis, and therefore $\Rs$-definable, as required.
\end{proof}
 
% (\textbf{possibly more details are needed??})

\begin{ex}
Let $(K, +, \cdot)$ be a real closed field definable in the o-minimal structure \M. Suppose 
$
f \colon K \times K \to K
$
is a definable $2$-cocycle with the respect to $+$. Define on $K^2$ the following operations:
\[
(x, a) \oplus (y, b) = (x + y, a+b+f(x, y))
\] 

\[
(x, a) \otimes (y, b) = (0, xy).
\]

\medskip
Then $R = (K^2, \oplus, \otimes)$ is a definable ring. Since $(K, +)$ is divisible, we know that $f$ is a coboundary, the definable group $(K^2, \oplus)$ is isomorphic to $(K, +)^2$
and $R$ is isomorphic to the nilpotent associative $K$-algebra
\[
\left \{ 
\begin{pmatrix}
0 & 0 & 0 \\
x & 0 & 0 \\
a & x & 0
\end{pmatrix} : x, a \in K
\right \}.
\]

However, we do not know whether there is a definable $2$-cocycle that is not \emph{definably} coboundary, meaning that none of the complements of the $1$-dimensional definable subgroup $\{0\} \times K$ in $(K^2, \oplus)$ is definable.  
 
If there is such ``bad'' cocycle, the corresponding ring $R$ is a definable ring and it is an associative $K$-algebra, but the scalar multiplication $K \times R \to R$ is not definable, which is why in 
 \Cref{theo:nilpotent} we have to distinguish the cases where $\mtf(R, +)$ is a direct sum of $1$-dimensional definable subgroups or not.

%However, we do not know whether there is an o-minimal structure \M\ defining such non-definable $K$-vector spaces. In any case, \Cref{theo:nilpotent} takes into account all possibilities.
\end{ex}

In the proof of the second case of \Cref{theo:nilpotent}  we claimed there are nilpotent definable rings of arbitrary dimension $n+1$ such that $I_ n$ has trivial multiplication and 
the co-dimension of $\Ann(R)$ in $I_n$ is arbitrary too. A general example is given below.

\begin{ex} \label{ex:In-trivial}
 Suppose $K$ is a definable real closed field in $\M$ and let $R$ be the nilpotent associative algebra over $K$ of dimension $2n + 1$ given by the $3n \times 3n$ matrices of the form

\[
 \begin{pmatrix}
    \begin{matrix}
        0 & 0 & 0 \\
        x & 0 & 0 \\
        a_1 & b_1 & 0 \end{matrix} & 0 & \dots  & 0 \\
& & \ddots  & \\
    0  & \dots & 0 & \begin{matrix}
        0 & 0 & 0 \\
        x & 0 & 0 \\
        a_n & b_n & 0 
    \end{matrix} 
    \end{pmatrix}   
\] 

\medskip
with $x, a_i, b_i\in K$.  We can take $I_{2n}$ to be the ideal of matrices with $x = 0$, which has trivial multiplication. Note that $\Ann(R)$ is the set of matrices where $x = b_i = 0$, so 
$\dim I_{2n} - \dim \Ann(R) = n$.    
\end{ex}

\begin{rem}
The ideal $R_0 \se \Ann(R)$ from \Cref{theo:nilpotent} may not be definable even when $\mtf(R, +)$ is a direct sum of $1$-dimensional definable subgroups. An example is given below. 
 \end{rem}

\begin{ex}
Let $\M$ be the real field. Set $R = \R^2 \times [1, e[$ with ring operations
\[
(a, x, u) \oplus (b, y, v) = 
\begin{cases}
(a+b, x+y, uv) & \mbox{if } uv < e \\
(a+b, x+y+1, uv/e) & \mbox{otherwise}.
\end{cases}
\] 

\[
(a, x, u) \otimes (b, y, v) = (0, ab, 0).
\]

Then $R_1 = \R^2 \times \{1\}$ is a definable associative algebra, its additive 
complements are contained in $\Ann(R)$ and none of them is definable in \M.  
\end{ex}
%==================================================
 %==================================================
%===================================================
\medskip
\section{Reduced rings}

At the opposite end of the spectrum of nilpotent rings are the rings in which $0$ is the only nilpotent element. Such rings are usually called \textbf{reduced}. That is, in a reduced ring, $x^2 = 0$ if and only if $x = 0$.
  
 It is well-known that a finite ring is reduced if and only if it is a direct product of division rings (which are in addition commutative by Wedderburn's little theorem). We generalise this result to the o-minimal setting by showing that every reduced definable ring is a direct product of definable division rings. Moreover, these division rings are definable already in the ring structure.

 \begin{theo}\label{theo:reduced}
Every definable reduced ring \Rs\ is a direct product of \Rs-definable division rings. More precisely, if \rng\ is a definable reduced ring, then \Rs\ is unital and, given $s = \dim R(1)$, there are 
definable real closed fields $K_1, \dots, K_s$ such that:

\begin{itemize}

\medskip
\item $K_i$ is not definably isomorphic to $K_j$ if $i \neq j$;

\medskip
\item $R^0$ is a direct product of rings $R^0 = R_1 \times \cdots \times R_s$, where for $i=1, \dots, s$ $R_i$ is \Rs-definable and definably isomorphic to
\[
K_i^{r_i} \times K_i(\sqrt{-1})^{p_i} \times \mathbb{H}(K_i)^{q_i}.
\] 

\medskip
for some $r_i, p_i, q_i \in \N$, where $\mathbb{H}(K)$ denotes the quaternion algebra over $K$.

\medskip
\item $R = F \times R^0$, where $F$ is a direct product of fields.

\end{itemize}
\end{theo}

\begin{proof}
Suppose \rng\ is a definable reduced ring.
First note that every torsion element of $(R^0, +)$ is nilpotent by \Cref{prop:subgr}, so $(R^0, +)$ must be torsion-free. If $R$ is not definably connected, $R = F \oplus R^0$ by \Cref{fact:divisible-split}, because $R^0$ is divisible (\Cref{fact:tf-divisible}), where $F$ is a finite reduced ring, hence
a direct product of finite fields. 
By \Cref{prop:subgr}, $F \se \Ann(R^0)$. It follows that $R = F \times R^0$ and
\[
R^0 = \bigcap_{x \in F} \Ann(x).
\]    

This shows that $R^0$ is \Rs-definable. Therefore, we can assume $R$ is definably connected.

We prove that $R$ is a direct product of $\Rs$-definable division rings by induction on $n = \dim R$. If $n = 1$, $\Rs$ is a real closed field (\Cref{prop:1-dim}). Assume $n > 1$.

%\medskip
If $R$ is a division ring, there is nothing to prove. Otherwise, by \Cref{lem:zero-unit} there are distinct non-zero $a, b \in R$, such that $ab = 0$.

Set $A = \Ann_1(b)$. Then $a \in A$ and $b \notin A$, because $R$ is nilpotent-free. So $1 \leq \dim A < n$ and, by induction hypothesis, $A$ is a direct product of definably connected $\mathcal{A}$-definable division rings. Let $e_1 \in A$ be its multiplicative identity. 

Set $B = \Ann_2(e_1) = \Ann_2(A)$. Then $b \in B$ and $e_1 \notin B$, so $1 \leq \dim B < n$. By induction hypothesis, $B$ is a direct product of definably connected $\mathcal{B}$-definable division rings. Let $e_2 \in B$ be its multiplicative identity. 

%\medskip
If $x \in A \cap B$, then $x^2 = 0$, so $A \cap B = \{0\}$. Set $I = A + B = A \oplus B$. It is easy to check that $I$ is closed under multiplication, and therefore a subring, because $A$ is a left ideal and $B$ is a right ideal.

We claim that $I = A \times B$. That is, we need to check that $e_1$ and $e_2$ are orthogonal idempotents. We already know that $e_1e_2 = 0$ by construction. Since $A$ is a left ideal, $e_2e_1 \in A$. On the other hand, $B$ is a right ideal, so $e_2e_1 \in B$. Hence $e_2e_1 \in A \cap B = \{0\}$, and $I$ is the product of $A$ and $B$.

%Since $B = \Ann_2(A) \subseteq \Ann_1(A)$, we have $B = \Ann(A)$. And from $A = \Ann_1(B) \subseteq \Ann_2(B)$, we have $A = \Ann(B)$. 

Let $e = e_1 + e_2$ be the multiplicative identity of $I$. Since $e$ is an idempotent element, if $e \neq 1$, then $e$ is a zero-divisor.

Suppose $re = 0$. Then $re_1 = -re_2$. Because $A$ is a left ideal, $re_1 \in A$. In particular, $re_1e_1 = re_1$. On the other hand, $e_2e_1 = 0$, so $-re_2e_1 = 0 = re_1$. Hence, $re = 0$ implies $re_2 = 0$ and $r \in \Ann_1(B) = A$. However, $re_1 = 0$ and $r \in A$ implies $r = 0$. Similarly, $er = 0$ implies $r = 0$ because $B$ is a right ideal. That is, $e$ is not a zero-divisor, and $e = 1$. 

Finally, we claim that $I = R$. Assume, by way of contradiction, $r \notin I$. We know that $r = 1r = (e_1 + e_2)r = e_1r + e_2r$
and $e_2r \in B \subset I$, so $e_1r \not\in I$. On the other hand, $re_1 \in A \subset I$, so it must be $e_1r \neq re_1$.
 
It follows that $e_1R$ is a subring properly containing $A$. Moreover, $e_1$ is a left zero-divisor, so $\dim e_1R < n$. By induction hypothesis, $e_1R$ is a direct product of division rings.  Because $e_1$ is the multiplicative identity of one of those division rings, $e_1 \in Z(e_1R)$. In particular, 
\[
e_1 (e_1r) = e_1r = (e_1 r)e_1 
\]

\medskip
so $e_1(r - re_1) = 0$ and $r - re_1 \in \Ann_2(e_1) = B \subset I$. As $re_1 \in A \subset I$, $r \in I$ too, contradiction. Therefore, $R = A \times B$. Note that both $A$ and $B$ are $\Rs$-definable, so $R$ is a direct product of $\Rs$-definable division rings, as claimed. 

By \cite{PS99}, every definable division ring $D$ is an associative division algebra over a definable real closed field $K$. That is, $D = K$ or $D = K(\sqrt{-1})$ or $D= \mathbb{H}(K)$.  
We are left to show that the number of non-isomorphic definable real closed fields appearing in the decomposition of $R$ is equal to $\dim R(1)$. We will prove it by induction on the number of division rings $R$ is decomposed in the product of. 

Suppose $R = D_1 \times \cdots \times D_n$, where $D_i$ is a \Rs-definable division ring over a definable real closed field $K_i$. We identify $D_i$ with $K_i$, $K_i(\sqrt{-1})$ or $\mathbb{H}(K_i)$. Let $\pi_i \colon R \to D_i$ be the canonical ring homomorphism. Note that 
$1 = (e_1, \dots, e_n)$, where $e_i$ is the multiplicative identity of $D_i$, so 
$R(1) \subseteq K_1 \times \cdots \times K_n$ and $\pi_i(R(1)) = K_i$ for each $i = 1, \dots, n$.

Suppose $n = 1$. In all three possible cases for $R = D$, $R(1) = K$ which is $1$-dimensional, as required.

Let $n > 1$. Define $S = D_2 \times \cdots \times D_n$ and let $e$ be its multiplicative identity. Set $s = \dim S(e)$. After reordering $D_2, \dots, D_n$  we can assume by induction hypothesis  that $S(e) = K_2 \times \cdots K_{s+1}$, where $K_i$ is not definably isomorphic to $K_j$ when $i \neq j$.

Note that $R(e) = \{0\} \times S(e)$ and $R(1) \subseteq K_1 \times S(e)$.  It follows that either $\dim R(1) = s$ or $\dim R(1) = s+1$.

Suppose $\dim R(1) = s$ and let $\pi \colon R \to S$ be the canonical ring homomorphism. 
Since $\pi(R(1)) = S(e)$ and $\dim R(1) = \dim S(e)$, the restriction of $\pi$ to $R(1)$ 
is a definable isomorphism with $S(e)$ by connectedness, as $(R^0, +)$ is torsion-free and every torsion-free definable group is definably connected. Since $R(1)$ projects surjectively onto $K_1$ in $D_1$, there is a definable ring embedding $i \colon K_1 \to K_2 \times \cdots \times K_{s+1}$
induced by $\pi$. If $i(K_1)$ maps nontrivially onto different $K_i$ and $K_j$, then they are definably isomorphic, contradiction. Therefore, there is exactly one $j \in \{2, \dots, s+1\}$ such that the projection of $i(K_1)$ onto $K_j$ is non-trivial, and so equal to $K_j$. As $K_1$ is definably isomorphic to $K_j$, the number of non-isomorphic definable real closed fields for $R$
is equal to $s = \dim R(1)$, as required.

Assume now $\dim R(1) = s + 1$. Then $R(1) = K_1 \times S(e)$. Suppose, by way of contradiction, $K_1$ is definably isomorphic to $K_i$ for some $i \in \{2, \dots, s+1\}$ and let $f \colon K_i \to K_1$ be a definable isomorphism. Then
\[
\{(f(\pi_i(x)), x) \in K_1 \times S : x \in S(e)\} 
\]

is a $s$-dimensional definable subring of $R$ containing $1$, contradiction.
\end{proof}    
%============
\begin{rem}
Unlike the nilpotent case, the real closed fields $K_1, \dots, K_s$ from \Cref{theo:reduced} are not necessarily $\Rs$-definable. For instance, take \Rs\ to be the complex field.
\end{rem}

\begin{rem}
For unital (reduced) ring it is not true that $R(1)$ has maximal dimension among monogenic subrings. For instance, $K(\sqrt{-1})$ is monogenic, as for any $x \neq 0$, $\la x \df$ is the straight line through the origin and $x$, which is a subring if and only if $x \in K$. Therefore, whenever $x \notin K$, $R(x) = R$.
\end{rem}

%============
\begin{cor}
Every $n$-dimensional definably connected reduced ring is elementarily equivalent to a $n$-dimensional reduced
associative $\R$-algebra.
\end{cor}

\begin{cor}
Let $\Rs = (R, +, 0, \cdot)$ be a definably connected ring in an o-minimal expansion of a real closed field $K$. If $\Rs$ is reduced, $\Rs$ is definably isomorphic to
 
\begin{equation*}
  K^r \times K(\sqrt{-1})^s \times \mathbb{H}(K)^p 
\end{equation*}

\medskip
for some $r, s, p \in \N$.
 \end{cor}

%======================================
 %======================================
 
 \section{The Jacobson radical} \label{section:Jacobson}

 We now consider rings that are not nilpotent nor reduced. If a ring is not reduced, then it contains some non-zero nilpotent element. However, $\{0\}$ may be its only \emph{nilpotent ideal}. When this is the case, the ring is called \textbf{semiprime}. Semiprime definable rings will be studied and fully characterized in the next section. In this section we show that every definably connected ring has a maximal nilpotent ideal (that happens to be definable) coinciding with its Jacobson radical and the quotient is semiprime.

The \textbf{Jacobson radical} $J(R)$ of a ring $R$ is the ideal consisting of the elements in $R$ that annihilate all simple right (and/or left) $R$-modules. Equivalently,
\[
J(R) = \{a \in R : \forall\, r \in R\ \exists\, b \in R \mbox{ such that } bra-ra-b = 0\}.
\]

The above characterization shows that $J(R)$ is $\Rs$-definable.  

Recall that a \textbf{minimal right ideal} of a ring $R$ is a non-zero right ideal such that $\{0\}$ is its only proper right ideal. Similarly for a \textbf{minimal left ideal}. They do not always exist. For instance, a domain that is not a division ring contains no minimal right ideal.
 
\begin{lem}\label{lem:min-principal}
Let $\Rs = (R, +, 0, \cdot)$ be a definably connected ring. Set
\[
X = \{aR : a \in R \} \qquad \mbox{and}\qquad Y = \{Ra : a \in R\}.
\]

\begin{enumerate}[(i)]
\item Suppose $\Ann_1(R) = \{0\}$. Every element in $X$ of smallest positive dimension is a minimal right ideal of $R$, and every minimal right ideal belongs to $X$.
 
\medskip
\item Suppose $\Ann_2(R) = \{0\}$. Every element in $Y$ of smallest positive dimension is a minimal right ideal of $R$, and every minimal right ideal belongs to $Y$.
\end{enumerate}
\end{lem}

\begin{proof}

\begin{enumerate}[(i)]
\item Suppose $I = aR \in X$ is of smallest positive dimension. Suppose $J \subseteq I$ is a right ideal and let $x \in J$, $x \neq 0$. Because $\Ann_1(R) = \{0\}$, $xR$ is nontrivial and  we have
$
xR \se J \se I = aR.
$
As $R$ is definably connected, so are $xR$ and $aR$. It follows that $xR = aR$ by minimality of $aR$. Thus $J = I$.

Suppose now $I$ is a minimal right ideal and let $x \in I$, $x \neq 0$. As $xR \subseteq I$, by minimality either
$xR = \{0\}$ or $xR = I$. In the first case, $x \in \Ann_1(R)$, contradiction. So it must be $xR = I$
and $I \in X$, as required.

\medskip
\item Similar to (i).
\end{enumerate}
\end{proof}

\begin{rem}
Not every minimal left/right ideal is of smallest possible dimension. For example, let 
$\M = (K, +, \cdot)$ be a real closed field. Set
\[
R = M_2(K) \times M_3(K).
\]
%\[
%S = \left \{
%\begin{pmatrix}
%a & 0 \\
%b & a
%\end{pmatrix} : a, b \in K
%\right  \} \quad \mbox{and}\quad R = S \times M_2(K).
%\]

\medskip
$R$ is a unital ring, so $\Ann_1(R) = \Ann_2(R) = \{0\}$. As usual, let $E_{i, j}$ denote the elementary matrix with entry $1$ in position $(i, j)$ and $0$ elsewhere. The minimal right ideal
\[
I = \{0\} \times \left \{a E_{1, 1} + b E_{1, 2} + c E_{1, 3}
 : a, b, c \in K
\right  \}
\] 

\medskip
is equal to $xR$ for $x = (0, E_{1, 1})$ and is $3$-dimensional, while $yR$ for $y = (E_{1, 1}, 0)$ is $2$-dimensional.
\end{rem}

\begin{rem}
When $R$ is unital it is obvious that every minimal right/left ideal is as in \Cref{lem:min-principal}. On the other hand, not every ring with trivial one-sided annihilator is unital.  
But we will crucially make use of  \Cref{lem:min-principal} to show that semiprime definable rings are unital.
\end{rem}

\begin{lem}\label{lem:idempotent-unit}
Let $R$ be a definably connected ring. If $e$ is a nonzero idempotent element, then $R(e)$ is a unital subring and $e$ is its multiplicative identity.
\end{lem}

\begin{proof}
Set $S = R(e)$. Write $S = F \oplus (A + B)$, where $F$ is finite, $(A + B) = S^0$, with $A$ definable torsion-free and $B$ the $0$-Sylow of $S$. Then $e = x + a + b$ for some $x \in F$, $a \in A$ and $b \in B$. By \Cref{prop:subgr}, the torsion subgroup of $(R, +)$ is contained in $\Ann(R)$ and $F+B \se \Ann(S)$. It follows that 
\[
e^2 = (x + a + b)^2 = a^2 = x + a + b = e.
\]   

Since $A$ is an ideal (\Cref{prop:subgr}), $a^2 - a - b = x \in F \cap S^0 = \{0\}$ and $b = a^2 - a \in A$. 
Therefore $e \in A$ and $S = A$ is torsion-free.  

Note that $eS = S$, because $e = e^2 \in eS$, so $eS$ is a definable subring containing $e$. 
Therefore, the annihilator of $e$ in $S$ is trivial, because all definable additive subgroups of $S$ are definably connected. For each $a \in S$, $e^2 a = ea$ so $e(ea - a) = 0$
and $ea = a$. That is, $e$ is a multiplicative identity for $S$, as claimed.
\end{proof}

\begin{fact}[Brauer's Lemma] \label{Brauer}
Any minimal right ideal $I$ in a ring $R$ satisfies $I^2 = \{0\}$ or $I = eR$ for some idempotent element $e \in R$.  Any minimal left ideal $I$ in a ring $R$ satisfies $I^2 = \{0\}$ or $I = Re$ for some idempotent element $e \in R$. 
\end{fact}

\begin{proof}[Proof of \Cref{theo:Jacobson}] Suppose $R$ is a $n$-dimensional definably connected ring and $J(R)$ its Jacobson radical. 

$(1) \Rightarrow (3)$. See \Cref{prop:nil-flag} and \Cref{cor:nilpotent-Rn+1=0}.

\medskip
$(3) \Rightarrow (2)$. Obvious.

\medskip
$(2) \Rightarrow (4)$. The Jacobson radical $J(R)$ contains every nil ideal of $R$ (See, for instance, Lam's book).

\medskip
$(4) \Rightarrow (5)$. Suppose $x \in R$ is an idempotent element. Since $R = J(R)$, using the definition of $J(R)$ given at the beginning of the section, for each $r \in R$ there is $b \in R$ such that  
$
rx = brx - b.
$ Take $r = x$. Then there is $b \in R$ with

\[
x^2 = x = bx - b.
\]

\medskip
Therefore, $x^2 = (bx - b)x = bx^2 - bx = bx - bx = 0$, and $x = 0$, as required.

 \medskip
$(5) \Rightarrow (1)$. We will prove by induction on $n = \dim R$ that if $R$ is not nilpotent, there is a nonzero idempotent.

If $n = 1$, then $R$ is a real closed field by \Cref{prop:1-dim}. Let $n > 1$. 

Assume $\Ann_1(R) \neq \{0\}$. By \Cref{prop:subgr}, $\Ann_1(R)$ cannot be finite, so the quotient $R/\Ann_1(R)$ is a definably connected ring of dimension smaller than $n$. Since $\Ann_1(R)$ is nilpotent and $R$ is not, $R/\Ann_1(R)$ is not nilpotent either. Indeed, if $I$ is a nilpotent ideal of $R$ with $I^k = \{0\}$ and $R/I$ is a nilpotent ring with $(R/I)^m = \{0\}$, then $R^{sm} = \{0\}$.

By induction hypothesis, 
$R/\Ann_1(R)$ contains some nonzero idempotent element $x+\Ann_1(R)$. Thus $x^2 - x \in \Ann_1(R)$. In particular,  
\[
(x^2 - x)x = x^3 - x^2 = 0.
\]

%\medskip
Therefore, $e = x^2$ is a nonzero idempotent in $R$. The case where $\Ann_1(R) = \{0\}$ and 
$\Ann_2(R) \neq \{0\}$ is similar.

Suppose now $\Ann_1(R) = \Ann_2(R) = \{0\}$. Then $(R, +)$ is torsion-free by \Cref{prop:subgr}. 

Let $a \in R$ such that $aR$ is of smallest dimension. By \Cref{lem:min-principal}, $aR$ is a minimal
right ideal. If $(aR)^2 \neq 0$, then $aR = eR$ for some idempotent element $e$ by Brauer's Lemma, and we are done.

Assume $(aR)^2 = 0$. Suppose $R$ is not commutative and set $I = \Ann_2(aR)$. Because $aR$ is a right ideal, $I$ is an ideal
and it is infinite, because it contains $aR$. If $I = R$, then $aR \subseteq \Ann_1(R) = \{0\}$, contradiction. Therefore, $I$ is a non-trivial proper definable ideal. If $I$ is not nilpotent, by induction hypothesis $I$ contains a non-zero idempotent element, as required. 

Assume $I$ is nilpotent and set $\ol R = R/I$. Note that $\ol R$ cannot be nilpotent, because $I$ is nilpotent and $R$ is not. By induction hypothesis, $\ol R$ contains a non-zero idempotent $\bar{u}$. Let $u \in R$ be in the pre-image of 
$\bar{u}$ and set $S = R(u)$. Since $\bar{u}$ is an idempotent in $\ol R$, $u^k \notin I$ for each $k \geq 1$. In particular, $u$ is not a nilpotent element and $S$ is not a nilpotent ring.

Because $R$ is not commutative, $R \neq S$ and by induction hypothesis $S$ has an idempotent element, as required.

Suppose now $R$ is commutative. Then $aR$ is an infinite proper nilpotent ideal. Set $\ol R = R/aR$. Because $\ol R$ is not nilpotent, by induction hypothesis it contains a non-zero idempotent $\bar{u}$ and we can assume $\ol R = \ol R (\bar{u})$.

By \Cref{fact:divisible-split}, there is an additive subgroup $B$ such that $R = aR \oplus B$. Let $b \in B$ in the pre-image of $\bar{u}$ and note that $R = aR + R(b)$. We can claim that $bR = R$.

If not, let $x \in \Ann(b)$ and write $x = \alpha + \beta$ with $\alpha \in aR$ and $\beta \in B$.
Then $bx = b \alpha + b \beta = 0$ and $b\beta \in aR$. On the other hand, 
by \Cref{lem:idempotent-unit} $\bar{u}$ is the multiplicative identity of $\ol R$, so $b\beta = \beta + \alpha'$
for some $\alpha' \in aR$. Thus $\beta \in B \cap aR = \{0\}$ and $\Ann(b) \se aR$.  
Because $aR$ is a minimal ideal, if $\Ann(b) \neq \{0\}$ it must be $\Ann(b) = aR$. In particular, $b \in \Ann(aR)$ and so $R(b) \subseteq \Ann(aR)$. But $aR$ has trivial multiplication, so $R = aR + R(b) \subseteq \Ann(aR)$ and $aR \subseteq \Ann(R)$, contradiction. Therefore $\Ann(b) = \{0\}$ and $bR = R$.

Let $e \in R$ such that $be = b$. Then $b(e^2 - e) = 0$ and $e \neq 0$ is an idempotent element. 

\medskip
Suppose now $R$ is not nilpotent. Since $J(J(R)) = J(R)$, the equivalence $(1)-(4)$ shows that $J(R)$ is nilpotent. Because $J(R)$ contains all nil ideals of $R$, in particular it contains all nilpotent ideals of $R$ and $R/J(R)$ is semiprime. Indeed, as noted above, if $I$ is a nilpotent ideal of $R/J(R)$ then its pre-image in $R$ is nilpotent too, because $J(R)$ is nilpotent, contradiction. The existence of a definable subring isomorphic to $R/J(R)$ such that $R = J(R) \oplus S$ will be proved in Section 7, after having   characterized semiprime rings in the next section.
\end{proof}

\begin{cor}\label{cor:commutative-Jacobson}
Let $R$ be a commutative definably connected ring. Then $J(R)$ coincides with the nilradical of $R$ and $R/J(R)$ is a reduced ring.
\end{cor}

 %========================================
 %========================================
  \section{Semiprime rings}
  
  As recalled in the previous section, a ring is called \textbf{semiprime} when $\{0\}$ is its only nilpotent ideal.  Equivalently, the subring $aRa \neq \{0\}$ whenever $a \neq 0$. For instance, 
$M_k(\C) \times M_n(\R)$ is semiprime for any positive integer $k, n$. The goal of this section is to fully characterize semiprime definable ring by proving \Cref{theo:simple} and \Cref{theo:semiprime}. To this end, we will first characterize the smaller subclass of prime definable rings and then we will show that semiprime definable rings are direct products of prime definable rings.

Recall that a ring is called \textbf{prime} if $aRb \neq \{0\}$ whenever $a$ and $b$ are nonzero. The typical example of a prime ring is $M_n(D)$, where $D$ is a division algebra. We will show that every definably connected prime ring in o-minimal structures is of this form.

\medskip
Let $R$ be a ring. Recall that a \textbf{non-trivial idempotent} is an idempotent element different from $0$ and $1$. One says that $\{e_i: i \in I\}$ is a set of \textbf{orthogonal idempotents} when each $e_i$ is a non-trivial idempotent and $e_ie_j = e_je_i = 0$ whenever $i \neq j$. Moreover, a \textbf{primitive idempotent} is a non-zero idempotent $e$ such that $eR$ is indecomposable as a right $R$-module; that is, such that $eR$ is not a direct sum of two nonzero submodules. Equivalently, $e$ is a primitive idempotent if it cannot be written as 
$e = a + b$, where $a$ and $b$ are non-zero orthogonal idempotents in $R$.

\begin{prop}\label{prop:semiprime-sum}
Let $R$ be a definably connected semiprime ring. Either $R$ is a division ring or there is a finite set of primitive idempotents $\{e_1, \dots, e_n\}$ such that  
 \[
R = Re_1 \oplus \cdots \oplus Re_n, \label{eq:semiprime-dec} \tag{$\ast$}
\]
 
 \medskip \noindent
 where each $Re_i$ is a minimal left ideal of $R$. 
\end{prop}

\begin{proof}
Note that whenever $R$ is a semiprime ring then
\[
\Ann_1(R) = \Ann_2(R) = \{0\},
\]

since the two one-sided annihilators have trivial multiplication, so they are nilpotent ideals. It follows that $(R, +)$ is torsion-free by \Cref{prop:subgr}.

By \Cref{lem:min-principal}, every principal left ideal $Ra$ of minimal dimension is a minimal left ideal. Fix one such $Ra$ of minimal dimension. Note that $Ra$ is not a null ring, otherwise the non-trivial $aRa$ is contained in $\Ann_2(R)$, contradiction. By Brauer's Lemma \Cref{Brauer}, $Ra = Re_1$ for some idempotent element $e_1$. 

If $Re_1 = R$, by minimality of $Re_1$ we have that for each non-zero $a \in R$, $\dim Ra = \dim Re_1 = \dim R$. It follows that $R$ has no zero-divisor and it is a division ring by \Cref{lem:zero-unit}. Suppose $Re_1 \neq R$ and set $J_1 = Re_1$.

We can now apply what is known as the left Pierce decomposition related to the idempotent $e_1$ (see \cite{Pierce}). Irrespective of $R$ being unital, one sets  
\[
J_2 = R(1-e_1) = \{x-xe_1: x \in R\}.
\]

\medskip
Note that $J_2$ is a left ideal too and $R = J_1 \oplus J_2$.

Suppose $J_2$ is a minimal left ideal. By Brauer's Lemma again, $J_2 = Re_2$ for some non-trivial idempotent $e_2$ and we are done.

If $J_2$ is not a minimal left ideal, let $I$ be a left ideal properly contained in $J_2$ and $a \in I$ non-zero. Then $Ra \subseteq I$ and $\dim Ra < \dim J_2$, because $(R, +)$ is torsion-free and all of its definable subgroups are definably connected. If $Ra$ is not minimal, let $I'$ be a left ideal properly contained in $Ra$ and take $b \in I'$ non-zero. Then $Rb \subseteq I'$ and $\dim Rb < \dim Ra$. If $Rb$ is not minimal, after a finite number of steps we get a minimal left ideal $Rx$. By Brauer's Lemma, $Rx = Re_2$
for some non-trivial idempotent $e_2$.  

We can now apply the left Pierce decomposition of $J_2$ related to the idempotent $e_2$ and write $J_2 = J_2e_2 \oplus J_3$, where $J_3 = J_2(1-e_2)$.

Because $J_2$ is a left ideal, both $J_2e_2$ and $J_3$ are left ideals too. Since $e_2 \in J_2e_2 \se Re_2$, $J_2e_2 = Re_2$ by minimality of $Re_2$ and 
\[
R = Re_1 \oplus J_2 = Re_1 \oplus Re_2 \oplus J_3.
\]

If $J_3$ is minimal, then $J_3 = Re_3$ for some idempotent $e_3$ and we are done. Otherwise, we repeat the process as above. If $\dim R = d$, after at most $d$ steps $J_n$ must be a minimal left ideal and \Cref{eq:semiprime-dec} holds. 

Since each $Re_i$ is a minimal left ideal, each $e_i$ is a primitive idempotent, as required.
 \end{proof}

\bigskip
A ring with non-trivial multiplication $R$ is called \textbf{simple} when $\{0\}$ and $R$ are its only ideals. We say that a definable ring $R$ is \textbf{definably simple} when 
$\{0\}$ and $R$ are its only \emph{definable} ideals.

\begin{proof}[Proof of \Cref{theo:simple}]
$(a) \Rightarrow (b).$ Obvious.

\medskip
$(b) \Rightarrow (c).$ If $R$ is definably simple, then it is definably connected, as $R^0$ is a definable ideal by \Cref{prop:subgr}. Since $\Ann(R) = \{0\}$, $(R, +)$ is torsion-free by \Cref{prop:subgr} and $R$ is not nilpotent by \Cref{prop:nil-flag}. 

If $R$ is not semiprime, then it contains a proper non-trivial nilpotent ideal $I$ and thus the Jacobson radical $J(R)$ is a proper non-trivial definable ideal by \Cref{theo:Jacobson}, leading to a contradiction. Therefore, $R$ is semiprime.

If $R$ is not prime, there are $a, b \in R$ non-zero such that $aRb = \{0\}$. Then $a \in \Ann_1(Rb)$ which is a definable ideal of $R$. Since $R$ is definably simple, it must be $R = \Ann_1(Rb)$. In particular, $bRb = \{0\}$ and $R$ is not semiprime, contradiction. It follows that $R$ is prime, as required.

\medskip
$(c) \Rightarrow (d).$ Suppose $R$ is definably connected and prime. If $R$ is a division ring, then $n = 1$ and there is nothing to prove. Suppose $R$ is not a division ring. By \Cref{prop:semiprime-sum}, $R$ can be decomposed as
\[
R = Re_1 \oplus \cdots \oplus Re_n, 
\]

where each $e_i$ is a non-trivial idempotent and $Re_i$ is a minimal left ideal. Recall that $(R, +)$ is torsion-free (see proof of \Cref{prop:semiprime-sum}), so every definable additive subgroup is definably connected. We claim that for each $i = 1, \dots, n$ the subring $e_iRe_i$ is a division ring 
with $1 = e_i$.

Clearly, $e_i$ is a multiplicative identity, because it is an idempotent element. Let $e_ixe_i \neq 0$. Therefore, $Re_ixe_i$ is a nonzero left ideal contained in $Re_i$. By minimality, $Re_ixe_i = Re_i$ and there is some $r \in R$ such that $re_ixe_i = e_i$. It follows that 
\[
(e_ire_i)(e_ixe_i) = e_i
\]

\medskip
and $e_ixe_i$ is invertible in $e_iRe_i$, as claimed.

Set $K_i = R(e_i) \subset e_iRe_i$. We know from \cite{PS99} that $K_i$ is a real closed field and $e_iRe_i$ is $K_i$-definable (see \Cref{theo:reduced}). It is easy to check that the definable map 
$K_i \times Re_i \to Re_i$  given by $(k, a) \mapsto ak$ provides a scalar multiplication making $Re_i$ a definable $K_i$-vector space.

Next, we need to show that all $K_i$'s are definably isomorphic to conclude that $R$ is a definable $K$-vector space, where $K = K_i$ for each $i = 1, \dots, n$.

Fix $i, j \in \{1, \dots, n\}$ and consider the subring $S = e_iRe_j$. Note that $S \neq \{0\}$, because $S$ is prime.
 %(\Cref{fact:prime}). 
 We claim that for each non-zero $x \in S$, $\la x \df$ is $1$-dimensional and definably isomorphic to both $(K_i, +)$ and $(K_j, +)$.

Write $x = e_ire_j$, with $r \in R$. As $x \in K_ire_j$ and the latter definable group is $1$-dimensional, it follows that $\la x \df = K_ire_j$. Similarly, $x \in e_irK_j$ and $\la x \df = e_irK_j$. Therefore, $K_ire_j = e_irK_j$ and our claim is proved.

Let $f \colon K_i \to K_j$ be the corresponding definable group isomorphism and note that $f(e_i) = e_j$. Fix a non-zero $x \in K_i$, let $y = f(x)$ and set 
\[
G = \{k \in K_i: f(kx) = f(k)y\}.
\]

\medskip
It is easy to check that $G$ is a definable subgroup of $(K_i, +)$ containing $e_i$. Thus $G = K_i$. As $x$ was an arbitrary element of $K_i$, it follows that $f$ is a ring isomorphism and $K_i$ is definably isomorphic (as a field) to $K_j$ for each $i, j \in \{1, \dots, n\}$.

Since $R$ is a definable $K$-vector space, it is a (finite dimensional) definable associative $K$-algebra by \Cref{prop:algebras}. By the celebrated Wedderburn's Theorem, every prime nonzero finite dimensional $K$-algebra $R$ is simple and it is $K$-definably isomorphic to $M_n(D)$, where 
$D = e_1Re_1$, for some $n \in \N$.  

\medskip
$(d) \Rightarrow (a).$ It is well-known that $M_n(D)$ is simple.
\end{proof}

%\begin{rem} (\textbf{just for me to rememeber})
%The real closed field $K$ in \Cref{theo:simple} is not necessarily $\Rs$-definable. For instance, take $\Rs$ to be the complex field. 
%\end{rem}

We extract the following immediate consequence of \Cref{theo:simple} for later use.

\begin{cor} \label{cor:prime-unital-rcf}
Let $R$ be a definably connected prime ring. Then $R$ is unital and $R(1)$ is a real closed field.
\end{cor}

Recall that a ring $R$ is called \textbf{semisimple} when it is Artinian and its Jacobson radical is trivial. Note that a simple ring is not necessarily semisimple because it may be not Artinian. The Weyl algebra is an example of such simple non-Artinian ring. However, \emph{definable} simple groups are semisimple by \Cref{theo:simple}.

 \begin{proof}[Proof of \Cref{theo:semiprime}]
 $(i) \Leftrightarrow (iii).$
If $R$ is semiprime, then $R$ has no nonzero nilpotent ideal. By \Cref{theo:Jacobson}, $J(R)$ is nilpotent, so clearly $J(R) = \{0\}$. Conversely, by \Cref{theo:Jacobson} $J(R)$ contains every nilpotent ideal of $R$, so if $J(R) = \{0\}$, then $R$ has no nonzero nilpotent ideal and it is semiprime.

\medskip
 $(i) \Leftrightarrow (iv).$
Suppose $R$ is semiprime. We want to show that $R$ is a direct product of simple $R$-definable rings. If $R$ is prime, then $R$ is simple by \Cref{theo:simple}, and there is nothing to prove. Suppose $R$ is not prime and let $I$ be a $R$-definable proper non-zero ideal of $R$, that we know exists by the proof of \Cref{theo:simple}. 
%We recall that we can assume $I$ is $R$-definable because if $R$ is not prime then the $R$-definable ideal $\Ann_1(Rb)$ is non-trivial for some $b \in R$. 
We claim that $I$ is semiprime.

If not, $J(I) \neq \{0\}$ by the equivalence $(i) \Leftrightarrow (iii)$ above. Although in general an ideal of an ideal is not necessarily an ideal of the ambient ring, this is the case for the Jacobson radical of an ideal. This is surely well-known, but as we could not find a reference, we add an easy argument below for completeness.

Let $x \in J(I)$ and $r \in R$. We want to show that $rx \in J(I)$. Because $I$ is an ideal, 
$rx \in I$. By the characterization of $J(I)$ given at the beginning of Section 5, we need to check that for each $s \in I$ there is $b \in I$ such that 
\[
bsrx - srx - b = 0.
\]
  
  \medskip
Since $sr \in I$ it follows that such $b$ exist because $x \in J(I)$. The other case $xr \in J(I)$ is similar by using the equivalent symmetric characterization
\[
J(R) = \{a \in R : \forall\, r \in R\ \exists\, b \in R \mbox{ such that } bar-ar-b = 0\}.
\]

But $R$ is semiprime and $\{0\}$ is its only nilpotent ideal, so $J(I) = \{0\}$ and $I$ is semiprime.  

 By induction hypothesis, $I$ is a direct product of simple 
 definable rings. By \Cref{theo:simple}, each simple direct factor is unital, hence $I$
is unital. Let $e \in I$ be its unity. 

Since $I$ is an ideal, clearly $Re \se I$. Conversely, $I = Ie \se Re$, so $I = Re$.

For each $x \in R$, $ex$ and $xe$ belong to $I$, so $ex = (ex)e$ and $xe = e(xe)$. It follows that $ex = xe$ and $e \in Z(R)$. In general, the left Pierce decomposition 
$R = Re \oplus R(1-e)$ recalled in the proof of \Cref{prop:semiprime-sum} can be refined into the two-sided Pierce decomposition 
\[
Re = eRe \oplus (1-e)Re \qquad R(1-e)= eR(1-e) \oplus (1-e)R(1-e),
\]

where 
\[
(1-e)Re = \{xe - exe : x \in R \}, \quad eR(1-e) = \{ex - exe : x \in R\}, 
\]
\[
(1-e)R(1-e) = \{x - ex - xe + exe : x \in R\}.
\] 
 
 \medskip
 However, since $e$ is central, $(1-e)Re = eR(1-e) = \{0\}$. Set $J = (1-e)R(1-e)$.
 Since $I = eRe \se \Ann(J)$, $J$ is an ideal too and $R = I \oplus J = I \times J$.  
 
 The same proof used above for $I$ shows that $J$ is semiprime too, so by induction hypothesis $J$ is a direct product of simple definable rings, and we are done.

Conversely, a direct product of simple rings is clearly semiprime.

\medskip
$(ii) \Leftrightarrow (iv).$
A ring is semisimple if and only if it is Artinian and its Jacobson radical is zero (find a reference). So if $R$ is semisimple, it is semiprime and a direct product of simple $R$-definable rings by the equivalences $(iii) \Leftrightarrow (i) \Leftrightarrow (iv)$ proved above. 

Conversely, suppose $R$ is a direct product of simple definable rings. Then every ideal of $R$ is a product of simple direct factors, hence definable. By DCC on definable subgroups, $R$ is Artinian and therefore semisimple.

\medskip
$(i) \Leftrightarrow (v).$ If $R$ is semiprime, by the equivalence $(i) \Leftrightarrow (iv)$ and \Cref{theo:simple},
$R$ is a direct product of simple definable rings $S_i$ of the form $S_i = M_{n_i}(D_i)$
and the required set of idempotents is the union of the canonical primitive idempotents of the $S_i$'s.

Conversely, if $R$ is not semiprime, then $J(R) \neq \{0\}$ by the equivalence 
$(i) \Leftrightarrow (iii)$. Let $a \in J(R)$. Then $aR \se J(R)$. Note that $aR \neq \{0\}$, because $R$ is unital. Moreover, $(R, +)$ is torsion-free, so if $aR$ is not minimal,  
we can take a smaller right ideal in $aR$ until we find a minimal principal right ideal by a dimension argument. As $\{e_iR : i = 1, \dots, n\}$ is the set of minimal right ideals of $R$, then $e_iR \se aR$ for some $i$ and $e_i \in J(R)$, in contradiction with $J(R)$ being nilpotent. Thus $J(R) = \{0\}$ and $R$ is semiprime.
\end{proof}

\begin{cor}
Every $n$-dimensional definably connected semiprime ring is elementarily equivalent to a $n$-dimensional semiprime associative $\R$-algebra.
\end{cor}

\begin{cor}
Every  $n$-dimensional definably connected semiprime ring in an o-minimal expansion of a real closed field $K$ is a $n$-dimensional semiprime associative $K$-algebra.
 \end{cor}
 
   %======================================
 %======================================
  %========================================
%========================================
 \section{The general case}
 
In this final section we study arbitrary definable rings. 
%and give proofs of our  main results \Cref{theo:fields}, \Cref{theo:field} and \Cref{theo:main}. 
We start by seeing that every definably connected ring with non-trivial multiplication defines an infinite field:

\begin{proof}[Proof of \Cref{theo:field}]
Let \rng\ be a definably connected ring with non-trivial multiplication. We will prove our claim by induction on $n = \dim R$.

If $n = 1$, $\Rs$ itself is a real closed field by \Cref{prop:1-dim}. Let $n > 1$.

If $R$ is simple, by \Cref{theo:simple} $R$ is definably isomorphic to $M_n(D)$ for some $n \in \N$, where $D$ is a definably connected division ring. Therefore, there is a non-trivial idempotent $e \in R$ such that $D = eRe$ is $\Rs$-definable. By \cite{PS99}, as recalled in the proof of \Cref{theo:reduced}, either $D$ is a field or $D = \mathbb{H}(K)$ for some definable real closed field $K$, where $\mathbb{H}(K)$ denotes the quaternion algebra over $K$.  In the first case, we are done. In the second case, let $a \in K$, $a \neq 0, 1$. Then $C(a) = K$ and $K$ is $\Rs$-definable. 

If $R$ is semiprime and not simple, by \Cref{theo:semiprime} there is an integer $k > 1$ and simple definable subrings $S_1, \dots, S_k$ such that 
\[
R = S_1 \times \cdots \times S_k.
\] 

\medskip
Let $e_i$ be the unity of $S_i$ for each $i = 1, \dots, k$. Then $S_1 = \Ann(e_2 + \cdots + e_k)$
is $\Rs$-definable. By the simple case, there is a $S_1$-definable infinite field $K$ and we are done.

If $R$ is not semiprime, by \Cref{theo:semiprime}, $J(R) \neq \{0\}$. If $R$ is not nilpotent, then $R/J(R)$ is a semiprime $\Rs$-definable ring by \Cref{theo:Jacobson}, and we are reduced to the semiprime case above.

Finally, if $R = J(R)$ is nilpotent, see \Cref{theo:nilpotent}.
 \end{proof}

 \begin{prop} \label{prop:2-dim}
Let \rng\ be a $2$-dimensional definably connected ring. Assume \Rs\ is not nilpotent nor reduced. There is a \Rs-definable real closed field $K$ such that 
\begin{enumerate}[(i)]
\item $\Rs$ is a definable associative $K$-algebra, or

\medskip
\item $\Rs = K \times R_0$, where $R_0$ is a $1$-dimensional $\Rs$-definable group with trivial multiplication. 
\end{enumerate}
\end{prop}

\begin{proof}
Assume first $(R, +)$ is torsion-free. We claim that either $R$ is commutative or $R$ is centerless. If not, $\dim Z(R) = 1$. Let $a \notin Z(R)$. Then $C(a)$ is a proper subring of $R$ and it must be $\dim C(a) = 1$, since $a \in C(a)$ and $R$ is definably connected. As $Z(R) \se C(a)$, it follows that $Z(R) = C(a)$ by connectedness, in contradiction with $a \in C(a)$ and $a \notin Z(R)$.   
 
Suppose $R$ is commutative.  Since $R$ is neither nilpotent nor reduced, its Jacobson radical $J$ must be $1$-dimensional by \Cref{theo:Jacobson}, and $R$ contains some nonzero idempotent $e \notin J$. By dimension reasons, it must be
 \[
 R = J + R(e),
 \]
 
 where $e$ is the unity of $R(e)$ by \Cref{lem:idempotent-unit}.
 The quotient $R/J$ is a $1$-dimensional semiprime ring, so a $\Rs$-definable real closed field $K$. Because $J$ is $1$-dimensional nilpotent, $J$ is a null ring. 
 
 If $Je = \{0\}$, then $J = \Ann(R)$ and $R = J \times R(e)$. Note that $R(e)$ 
 is definably isomorphic to $K$, so $(ii)$ holds. (If $(J, +)$ is definably isomorphic to $(K, +)$, 
 $(i)$ holds as well).
 
 If $Je \neq \{0\}$, it must be $Je = J$, as $J$ is a $1$-dimensional ideal. 
 Hence $eR = R$ and $R$ is unital and $1 = e$. As $\dim J = 1$, every element $ a \in J$ is $2$-nilpotent (\Cref{prop:bound}). Thus
$(1+a)(1-a) = 1$ and $(1+a)$ is a unit. It follows that the map $R \to R$, $r \mapsto (1+a)r$ is a definable automorphism of $(R, +)$ for each $a \in J$, $R$ is a direct sum of $1$-dimensional definable 
subgroups by \cite{Co21}, and $(i)$ holds by \Cref{prop:algebras}.

 Suppose $R$ is centerless. Let $e \in R$ be a nonzero idempotent element from \Cref{theo:Jacobson}. Because $R$ is not commutative, $R \neq R(e)$, so $\dim R(e) = 1$, $R(e)$ is a real closed field $K$
 definably isomorphic to the $\Rs$-definable $R/J$, and $R = J \oplus R(e) = J \oplus K$. Because $J$
 is an ideal, $JK$ and $KJ$ are both contained in $J$, so they must be equal to either 
 $\{0\}$ or $J$. 
 
Because $R$ is centerless, they cannot both be trivial, otherwise $J \se \Ann(R) \se Z(R)$. So it must be
$JK = J$ or $KJ = J$. In either case, $J$ is a $1$-dimensional $K$-vector space, and $(i)$ holds by \Cref{prop:algebras}.  
 
Suppose now $G = (R, +)$ is not torsion-free. We claim that in this case 
 $(ii)$ holds. The infinite $0$-Sylow subgroup $R_0$ is contained in $\Ann(R)$ by \Cref{prop:subgr}.  As $\Rs$ has non-trivial multiplication, $\dim R_0 = 1$. Therefore $\dim \mtf(G) = 1$ and $\Rs = \mtf(G) \times R_0$. Because $\Rs$ has non-trivial multiplication, $\mtf(G)$ is a real closed field by \Cref{cor:null-tf} and \Cref{prop:1-dim}. Thus $R_0 = \Ann(R)$ is $\Rs$-definable. To see that $\mtf(G)$ is $\Rs$-definable too, note that $\mtf(G) = aR$ for each $a \notin \Ann(R)$.  
 \end{proof}

  %%%%===============================
  %===================================
\begin{proof}[Proof of \Cref{theo:main}]
By induction on $n = \dim R$. If $n = 1$, $R = K_1$ is a real closed field and $s = 1$.
Suppose $n > 1$. We will first prove (1)-(4) distinguishing the three cases $R$ nilpotent, $R$ semiprime and $R$ neither nilpotent nor semiprime, and afterwards we will prove $(5)$. 

If $R$ is nilpotent, all claims have been proved in \Cref{theo:nilpotent}.

If $R$ is semiprime, $\Ann(R) = \{0\} = R_0$. By \Cref{theo:semiprime} and \Cref{theo:simple}, $R$ is a direct product of simple definable ideals, each of the form $M_n(D)$ where $D$ is a $K$-definable division ring, for some definable real closed field $K$. Let $K_1, \dots, K_s$  be the non-isomorphic real closed fields appearing in the simple factors above, and set $R_1, \dots, R_s$ to be the corresponding ideals, so that 
$R = R_1 \times \cdots \times R_s$. Each $R_i$ is a finite-dimensional definable associative $K_i$-algebra, and $R$ is unital. Note that 
$R(1) = R_1(1) \times \cdots \times R_s(1) = K_1 \times \cdots \times K_s$ and $s = \dim R(1)$, so $(3)$ and $(4)$ holds.
  
%  \bigskip
Assume now $R$ is not nilpotent nor semiprime. In this case, the Jacobson radical $J = J(R)$ is an infinite nilpotent ideal and the quotient $R/J$ is a semiprime ring. Together with $(1)-(4)$ we will see there is a semiprime definable subring $S$ definably isomorphic to $R/J$ such that 
\[
R = J \oplus S,
\] 

\medskip
completing the proof of \Cref{theo:Jacobson}. The other claims were proved in \Cref{section:Jacobson}.

%\newpage
Assume \fbox{$J^2 = \{0\}$}. That is, $J$ is a null ring. Moreover, suppose first
\fbox{$ \dim R/J = 1$}, so that $R/J$ is a real closed field. If $\dim J = 1$, then $\dim R = 2$ and our claims follow from \Cref{prop:2-dim}.  Let $\dim J > 1$. We need to show   there is a definable subring $K$ that is a real closed field definably isomorphic to $R/J$ such that
\[
R = J \oplus K = R_0 \times R_1,
\] 

\medskip
where $R_1$ is a definable associative $K$-algebra and $R_0 \subseteq \Ann(R)$.

\begin{claim}\label{claim:unital}
If $R$ is unital and $R(1)$ is a real closed field $K$, then $R$ is a definable associative $K$-algebra.
\end{claim}

\begin{proof}
The restriction of the ring multiplication $K \times R \to R$, $(k, r) \mapsto kr$ provides a scalar multiplication making $R$ a definable $K$-vector space. By \Cref{prop:algebras}, $R$ is a definable associative $K$-algebra.  
\end{proof}

 \medskip
Suppose $R$ is unital.  For each nonzero 
$x \in J$, $J = \Ann(x) = \Ann_1(x) = \Ann_2(x)$, as $J$ is a null ring and $1 \notin \Ann_i(x)$, so $\Ann(x) \neq R$. 

It follows that $\dim xR = \dim Rx = 1$. Since $x \in xR \cap Rx$, it must be 
\[
R(x) = xR = Rx.
\]

In particular, $R(x) \subsetneq J$ is a definable ideal. Because $J/R(x)$
is the Jacobson radical of $R/R(x)$, by induction hypothesis 
\[
R/R(x) = (J/R(x)) \oplus \ol K,
\]

\medskip
 where $\ol K = (R/R(x))(1)$ is a real closed field definably isomorphic to $R/J$. Let $A$ be the pre-image of $\ol K$ in $R$. Note that because $1 \in R$ maps to the unity of $R/R(x)$ coinciding with the unity of $\ol K$, $1 \in A$.

By induction hypothesis, $A = R(x) \oplus S$, where $S = K = R(1)$ is a real closed field definably isomorphic to $\ol K$ such that  $R = J \oplus R(1)$, and $R$ is a definable associative $K$-algebra by \Cref{claim:unital}.

\medskip
Assume now $R$ is not unital. Because $R$ is not nilpotent, by the equivalence $(1) \Leftrightarrow (5)$ in \Cref{theo:Jacobson}, $R$ contains some idempotent element $e \notin J$. Note that $e + J$ must be the unity of $R/J$, that is the only nonzero idempotent element in a real closed field. It follows that for each $x \in R$, the elements $ex$ and $xe$ map to $x + J \in J/R$. Hence $eR$ and $Re$ map surjectively onto $R/J$ and 
\[
R = J + eR = J + Re.
 \] 
   
   \medskip
 Because $R$ is not unital, $e$ is a zero-divisor by \Cref{lem:zero-unit}. It follows that at least one of $eR$ or $Re$ has smaller dimension than $R$. Call it $A$.

By induction hypothesis, there is a definable subring $K$ of $A$ definably isomorphic to $R/J$ such that $A = J(A) \oplus K$. Hence $R = J \oplus K$ and $R(e) = K$ is a real closed field with unity $e$.

\medskip
We can see that $eR + Re$ is a definable $K$-vector space. To this end:

\begin{claim} \label{claim:eR+Re}
Let $e$ be an idempotent element such that $R(e)$ is a real closed field $K$. 

\begin{itemize}
\item If $\dim eR = n$, there are
$r_1, \dots, r_n \in R$ such that $eR = Kr_1 \oplus \cdots \oplus Kr_n$.

\medskip
\item If $\dim eR = n$, there are
$r_1, \dots, r_n \in R$ such that $Re = r_1K \oplus \cdots \oplus r_nK$.

\medskip
\item If $\dim (eR + Re) = m$, there are $r_1, \dots, r_m \in R$ such that 
\[
eR + Re = Kr_1 \oplus \cdots \oplus Kr_n \oplus r_{n+1}K \oplus \cdots \oplus r_{m}K.
\]
\end{itemize}
\end{claim}

\begin{proof}
We claim that $(eR, +)$ is torsion-free. If not, suppose $ea = b \neq 0$ is a torsion element.
By \Cref{prop:subgr}, $b \in \Ann(R)$. However, $eb = e^2a = ea = b \neq 0$, contradiction.

Let $\dim eR = n$. Write $x_1 = er_1 \in eR$, $x_1 \neq 0$. Then $x_1 \in Kr_1$ and 
$\dim Kr_1 = 1$, so $\la x_1 \df = Kr_1$. 

If $n = 1$, we are done. Otherwise, let $x_2 \in eR$, $x_2 \notin Kr_1$.  Similarly, 
$\la x_2 \df = Kr_2$ for some $r_2 \in R$. Note that $Kr_1 \neq Kr_2$ implies that $Kr_1 \cap Kr_2 = \{0\}$.  

If $n = 2$, we are done. Otherwise, let $x_3 \in eR$, $x_3 \notin (Kr_1 \oplus Kr_2)$
and $\la x_3 \df = Kr_3$ with $Kr_3 \cap (Kr_1 \oplus Kr_2) = \{0\}$. We iterate until we reach $n$.

The subring $Re$ is similar to $eR$, and $eR + Re$ follows.
\end{proof}

If $\Ann_1(e) = \{0\}$, then $Re = R$ is a definable associative $K$-algebra by \Cref{claim:eR+Re}. If $\Ann_1(e) = \Ann_1(R) \neq \{0\}$ and $xe \in \Ann_1(R)$, then
$xe^2 = xe = 0$ and $x \in \Ann_1(e)$. It follows that 
\[
\Ann_1(R/\Ann_1(R)) = \{0\} \quad \mbox{and} \quad R = Re + \Ann_1(R).
\]

\medskip
Similarly,
\[
\Ann_2(R/\Ann_2(R)) = \{0\} \quad \mbox{and} \quad R = eR + \Ann_2(R).
\]

Therefore,
\[
R = Re + eR + \Ann(R).
\]

Moreover, note that $eR + Re$ is an ideal, because $R = J \oplus R(e)$ and 
$J^2 = \{0\}$. Set $R_1 = eR + Re$. If $R_1 \neq R$, let $R_0$ be an additive subgroup of $\Ann(R)$ from \Cref{fact:divisible-split} such that 
$\Ann(R) = (\Ann(R) \cap R_1) \oplus R_0$. Then
\[
R = R_0 \oplus R_1 = R_0 \times R_1
\]

\medskip
and $(2)$ and $(3)$ hold. Assume now \fbox{$\dim R/J  > 1$}. By the semiprime case, 
\[
R/J = \ol S_1 \times \cdots \times \ol S_d
\]

\medskip
where each $\ol S_j$ is a unital finite-dimensional \emph{definable} associative $K_j$-algebra and $K_i$ is not definably isomorphic to $K_j$ when $i \neq j$.

When $d = 1$, we need to show again that $R = R_0 \times R_1$, where $R_1$ is a definable associative $K$-algebra and $R_0 \se \Ann(R)$.

By the $1$-dimensional case above, the pre-image of $(R/J)(1)$ in $R$ is of the form
$J \oplus K$, with $K$ definably isomorphic to $(R/J)(1)$.  

If $R$ is unital, then $1 \in J \oplus K$, so $R(1) = K$ and $R = R_1$ is a definable associative $K$-algebra by \Cref{claim:unital}. By Wadderburn principal theorem (see Theorem 2.5.37 in \cite{Rowen}), there is a $K$-subalgebra $S$
isomorphic to $R/J$ such that $R = J \oplus S$. Let $m$ be the dimension of $S$ and 
$\{x_1, \dots, x_m\}$ be a basis of $S$ over $K$. Then $S = \oplus_{i = 1}^m Kx_i$ is definable.

If $R$ is not unital, then the unity $e$ of $K$ is a zero-divisor, and as shown above,
\[
R = J + eR = J + Re
\]

\medskip
because $e + J$ is the unity of $R/J$ again. Because at least one of $eR$ and $Re$
has smaller dimension than $R$, by induction hypothesis we can find a definable semiprime ring $S$ that is definably isomorphic to $R/J$ and such that
\[
R = J \oplus S.
\]

As before, $eR + Re$ is a definable associative $K$-algebra and 
$\Ann_i(e) = \Ann_i(R)$ for $i = 1, 2$. Thus the same proof as above applies to show
$(2)$ and $(3)$ hold.     

\medskip
  Suppose now $d > 1$. When dealing with several real closed fields, the following will be useful:

\begin{claim}\label{claim:group-enough}
Let $K_1$, $K_2$ be real closed fields with unity $e_1$ and $e_2$ respectively. A definable homomorphism $f \colon (K_1, +) \to (K_2, +)$ is a ring homomorphism if and only if $f(e_1) = e_2$.
\end{claim}

\begin{proof}
The condition is obviously necessary. To see that it is sufficient,
fix $a \in K_1$ and let $b = f(a) \in K_2$. Set
\[
G_a = \{x \in K_1: f(ax) = b f(x)\}.
\]

It is easy to check that $G_a$ is a definable subgroup of $(K_1, +)$ containing $e_1$. Since
$(K_1, +)$ has no proper non-trivial definable subgroup, it follows that $G_a = K_1$. As $a$ is arbitrary, $f$ is a ring homomorphism, as claimed. 
\end{proof}

By the case $d = 1$ the pre-mage of each $\ol S_j$ in $R$ is of the form $J \oplus S_j$, where $S_j$ is a semiprime definable subring which is definably isomorphic to $\ol S_j$
and, therefore, a unital definable associative $K_j$-algebra. 

Let $e_j$ be the unity of $S_j$, so that $K_j = S(e_j) = R(e_j)$. 

\begin{claim} \label{claim:different-orthogonal}
If $i \neq j$, then $S_iS_j = S_jS_i = \{0\}$. That is, $S_i \se \Ann(S_j)$. 
\end{claim}

\begin{proof}
It suffices to show that $e_i$ and $e_j$ are orthogonal idempotent. In any case, 
$e_ie_j \in J$, because $e_i + J$ and $e_j + J$ are orthogonal in $R/J$.

If $e_ie_j = a \neq 0$, the $1$-dimensional additive subgroup $K_ie_j$ contains $a$, hence 
$\la a \df = K_ie_j$. Similarly, $\la a \df = e_iK_j$. As $K_ie_j = e_iK_j$, there is a definable isomorphism $(K_i, +) \to (K_j, +)$ such that $e_i \mapsto e_j$. By 
\Cref{claim:different-orthogonal}, $K_i$ is definably isomorphic to $K_j$ as a field, contradiction. Therefore, $e_ie_j = e_je_i = 0$.
\end{proof}

As $\{e_1, \dots, e_n\}$ is a set of orthogonal idempotents, $e = e_1 + \dots + e_n$
is an idempotent too, mapping to the unity of $R/J$. Set $S = S_1 \oplus \cdots \oplus S_d$. By \Cref{claim:different-orthogonal},
\[
S = S_1 \times \cdots \times S_d
\]

\medskip
is a definable subring of $R$ which is definably isomorphic to $R/J$ and such that 
$R = J \oplus S$. If $R$ is unital, $1 = e + a$ with $a \in J$. Because $J$ has trivial multiplication,

\[
1a = (e + a)a = ea = a \quad \mbox{ and } \quad e1 = e(e + a) = e + ea = e.
\]

\medskip
It follows that $a = 0$ and $1 = e \in S$.  Clearly, $R(1) = K_1 \times \cdots \times K_d$ and $s = d$.

\medskip
By the case $d = 1$, for each $i = 1, \dots, s$
\[
J \oplus S_i = R_0^i \times R_i
\]

\medskip
where $R_i$ is a definable associative $K_i$-algebra of the form $J_i \oplus S_i$ 
($J_i = J(R_i) = R_i \cap J$) and $R_0^i \se \Ann(J \oplus S_i)$.

\begin{claim}\label{claim:different-orthogonal-algebras}
If $i \neq j$, $R_i \se \Ann(R_j)$.
\end{claim}
  
 \begin{proof}
 This is similar to \Cref{claim:different-orthogonal}, so we will be brief. The main observation is that $R_i$ is a definable $K_i$-vector space, while $R_j$ is a definable $K_j$-vector space.   
 
 If $R_i$ is not included in $\Ann(R_j)$, let $a \in R_i$ and $b \in R_j$ such that $ab = c \neq 0$. Set $A = \la a \df$, 
 $B = \la b \df$, $C = \la c \df$. Then $Ab = aB = C$ and by \Cref{claim:group-enough}, $K_i$ is definably isomorphic to $K_j$, contradiction. 
 \end{proof} 
 
Set $R' = R_1 \times \cdots \times R_s$. If $R = R'$ then $R_0$ is trivial and $(2)-(4)$ hold. Otherwise, we can find a direct additive complement $R_0$ of $R'$ in $R$ in
\[
\bigcap_{i = 1}^s R_0^i \se \bigcap_{i = 1}^s \Ann(J \oplus S_i) \se 
\bigcap_{i = 1}^s \Ann(S_i) = \Ann(S) = \Ann(R).
\] 

This shows that $(2)$ and $(3)$ hold. If $R$ is unital, $\Ann(R)$ is trivial and so is $R_0$, and $(4)$ holds too.  
  
 \bigskip
Assume now \fbox{$J^2 \neq \{0\}$}. That is, $J$ is not a null ring. By \Cref{prop:subgr} and \Cref{prop:nil-flag}, $\Ann(J)$ is an infinite ideal. 

Let $\ol R = R/\Ann(J)$ and $\ol J = J/\Ann(J)$. Note that $J(\ol R) = \ol J$. By induction hypothesis, 
there is a semiprime definable subring $\ol S$ of $\ol R$ which is definably isomorphic to $\ol R/\ol J$, and $\ol R = \ol J \oplus \ol S$.  

Let $A$ be the pre-image of $\ol S$ in $R$. Then $J(A) = \Ann(J)$ and 
$A/J(A) = \ol S$. By induction hypothesis, $A = \Ann(J) \oplus S$ for some semiprime definable subring $S$ which is definably isomorphic to $A/\Ann(J) = \ol S = R/J$ and $R = J \oplus S$. The proof of \Cref{theo:Jacobson} is now completed.  By the semiprime case,
\[
S = S_1 \times \cdots \times S_d
\]

\medskip
where each $S_j$ is a unital finite-dimensional \emph{definable} associative $K_j$-algebra. Let $e_j$ be the unity of $S_j$ and $e = \sum_{i = 1}^d e_i$ the unity of $S$.
Because $J$ is not a null ring, 
\[
J = J_0 \times J_1 \times \cdots \times J_p,
\]

\medskip
where for $i > 0$, $J_i$ is a finite-dimensional associative $\ol K_i$-algebra that is not a null ring. So we have
\[
R = J \oplus S = (J_0 \times J_1 \times \cdots \times J_p) \oplus (S_1 \times \cdots \times S_d).
\]

\medskip
By similar arguments used in \Cref{claim:different-orthogonal}, $J_i \subseteq \Ann(S_j)$, whenever $\ol K_i$
is not definably isomorphic to $K_j$. On the other hand, if $\ol K_i$ is definably isomorphic to $K_j$, then
\[
R_j = (J_ie_j + e_jJ_i + J_i) \oplus S_j = J_i + e_iR + Re_i 
\]

\medskip
 is an associative $K_j$-algebra with Jacobson radical $J_ie_j + e_jJ_i + J_i$ that is orthogonal to any other $R_k$ obtained similarly.
 
 Assume $d = 1$ so that $S$ is a definable associative $K$-algebra with unity $e$. If one of the $\ol K_i$'s is definably isomorphic to $K$, suppose is $\ol K_p$ to ease the notation. 
 
 If this is not the case, set $s = p+1$ and $R_i = J_i$ for $i = 1, \dots, p$. The additive subgroup $(eR + Re) \cap J_0$ has a complement $H$ in $J_0$. Moreover, the additive subgroup $H \cap \Ann(e)$  has a complement $A$ in $H$. Note that $Ae + eA$ is a not trivial additive subgroup contained in $Re + eR$. So now define 
 \[
 R_{p+1} = R_s = eR + Re + A
 \]
 
and 
 \[ 
 R_0 = H \cap \Ann(e) \subseteq \Ann(J) \cap \Ann(S) = \Ann(R).
 \]
 
 \medskip
 If $\ol K_p$ is definably isomorphic to $K$, set $s = p$, $R_i = J_i$ for $i = 1, \dots, p-1$,  
 \[
 R_p = R_s = eR + Re + A + J_p
 \]
 
 \medskip
 and $R_0 = H \cap \Ann(e) \subseteq \Ann(R)$, where $A$ and $H$ are as above.
 
 If $d > 1$, one can first consider the definable subrings $J \oplus S_j$ using the case $d = 1$, and by orthogonality of associative algebras over non-isomorphic real closed fields, we get the general case
 where $s$ is equal to $p$ plus the number of $K_j$'s that are not definably isomorphic to any $\ol K_i$'s. 
 
 \bigskip
 If $\mtf{(J, +)}$ is a direct sum of $1$-dimensional definable subgroups, then each $J_i$ is a definable 
 $\ol K_i$-vector space, and so is each $R_i$ as a consequence. By \Cref{prop:algebras}, (3) holds.
 
 \bigskip
 If $R$ is unital, $R_0$ is trivial and each $R_i$ is unital with unity $e_i$ and $1 = \sum e_i$. 
 
 If $s = 1$, then $R$ is a definable associative $K$-algebra by \Cref{claim:unital}. If $s > 1$,
 let $\ol e_j$ be the sum of all the $e_i$'s except for $e_j$. Then $R_j = \Ann(\ol e_j)$ is a definable associative $K_j$-algebra with $R_j(e_j) = K_j$ and (4) holds.

\bigskip
 Finally, let us show that (5) holds. In general, when a ring $R$ is not unital, there is a natural unital ring, defined over $\Z \times R$, that is unital, contains $R$ as an ideal and it is the smallest such. However, if $R$ is definable in an o-minimal structure, then $\Z \times R$ surely is not, and we would like to find a \emph{definable} unital ring containing $R$ as an ideal. 
 
 By (2)--(4), every definable subgroup of the additive group of a unital definable ring is a direct sum of $1$-dimensional vector spaces over real closed fields, and the conditions for $R$ to embed as an ideal in a definable unital ring are therefore necessary.

To see that they are sufficient, assume first $R_0$ is trivial. Given $R = R_1 \times \cdots \times R_s$, after reordering the direct factors, let $R_1, \dots, R_p$ be the ones that are not unital. For $i = 1, \dots, p$, let $R_i^{\wedge} = K_i \times R_i$ with usual addition and multiplication given by
\[
(a, x)(b, y) = (ab, ay + bx + xy).
\]

\medskip
Then $R_i^{\wedge}$ is a unital definably connected ring containing $R_i$ as an ideal. 
As $\dim R_i^{\wedge} = \dim R_i + 1$, $R_i^{\wedge}$ is minimal. Set $K = K_1 \times \cdots \times K_p$ 
and $R^{\wedge} = K \times R$ with usual addition and the obvious multiplication from the $R_i^{\wedge}$s. Then $R^{\wedge}$ is the smallest unital definably connected ring containing $R$ as an ideal.

If $R_0$ is not trivial, then $K$ must be replaced with $K \times A$, where $A$ is the direct sum of the $A_i$'s that are not pairwise definably isomorphic as groups, acting on $R_0$ in the obvious way, so completing the proof of \Cref{theo:main}.
\end{proof}
 
 \medskip
When a definable ring is not definably connected, one can reduce to finite rings and definably connected rings due to the following:

\begin{proof}[Proof of \Cref{prop:disconnected}]
 The $n$-torsion subgroup $T_n$ of $(R, +)$ is a finite ideal contained in $\Ann(R^0)$ by \Cref{prop:subgr}, and $(R^0)^2 \se \mtf(R, +)$. Since $R = T_n + R^0$, $R^2 \se T_n \times \mtf(R, +)$, as claimed.
 
 If $R$ is unital, write $R = F \oplus R^0$ and $1 = a + b$ with $a \in F$ and $b \in R^0$. Since $a \se \Ann(R^0)$ and $b \in \Ann(F)$, it follows that $a$ and $b$ are unities for $F$ and $R^0$ respectively. 
 Therefore, $(R^0, +)$ is torsion-free by \Cref{theo:main}, $F = T_n$, and $R = F \times R^0$, as required.
   \end{proof}
 
 \bigskip
 \begin{proof}[Proof of \Cref{theo:fields}]
Let $\M = (M, <, +, \cdot, \dots)$ be an o-minimal expansion of a field. We first notice that the rings in $(a)$ and $(b)$ are $\M$-definable.

\medskip
For $(a)$, every finite-dimensional associative $M$-algebra is isomorphic, as a ring, to a subalgebra of $M_n(M)$, the ring of $n \times n$ matrices with entries in $M$, for some positive integer $n$. Every subalgebra of $M_n(M)$ is \M-definable, after fixing a basis.

\medskip
 For $(b)$, suppose $(R, + , 0, \cdot)$ is a finite ring, $(G, \oplus, 0)$ is a definable abelian group and $i \colon (R, +) \to (G, \oplus)$ is an  embedding such that $i\inv (G^0) \subseteq \Ann(R)$.   
 
 Set 
$R' = i(R) \subseteq G$. For $x, y \in G$ define
\[
x \otimes y = 
\begin{cases}
i(i\inv (x) \cdot i\inv (y)) & \mbox{ if } x, y \in R' + G^0 \\
0 & \mbox{otherwise}.
\end{cases}
\]

One can check that $(G, \oplus, \otimes, 0)$ is a \M-definable ring as described.  

\medskip
Conversely, suppose \rng\ is a \M-definable ring. Assume first $R$ is definably connected. Since all \M-definable real closed fields are isomorphic to $(M, +, \cdot)$,  by \Cref{theo:main}, $R = R_1 \times R_0$, where $R_1$ is an associative $M$-algebra and 
$R_0 \se \Ann(R)$. Note that even if $R_1$ or $R_0$ is not definable, they are both isomorphic to definable rings as in $(a)$ and $(b)$ respectively, where the finite ring in $(b)$ is just the trivial ring.

Suppose $R$ is not definably connected. By the connected case, $R^0 = A \times B'$, where 
$A$ is an associative $M$-algebra and $B \se \Ann(R^0)$.

Since $A$ is a divisible subgroup of the abelian $(R, +)$, $A$ is a direct summand. Let $B$ be such that $R = A \oplus B$.

If $|R/R^0| = n$, by \Cref{prop:disconnected}, the $n$-torsion subgroup $T_n$ is a (finite) ideal.  
As $A$ is torsion-free, the $n$-torsion subgroup of $R$ coincides with the $n$-torsion subgroup of $B$, and therefore $T_n \se B$. Moreover, $B \se \Ann(A)$ by \Cref{prop:subgr}, so $R = A \times B$, and we are done.
 \end{proof}
 
 When $R$ is not unital, the definably connected component $R^0$ may not be a direct factor, and the finite ring in the ideal $B$ from \Cref{theo:fields} may have non-trivial intersection with $R^0$. Below is an example.
 
 \begin{ex}\label{ex:disconnected}
Let $\mathcal{A} = (A, < , +, 0)$ be and ordered divisible abelian group. Fix $a \in A$, $a > 0$. 
On $[0, a[$ let $\oplus_a$ be the addition modulo $a$. Set $R = \Z_2 \times [0, a[$ with operations
\[
(t, x) \oplus (s, y) = (t+s,\ x \oplus_a y)
\]
\[
(t, x) \otimes (s, y) =  
\begin{cases}
(0, a/2) & \mbox{ if } t=s=1 \\
(0, 0) & \mbox{otherwise}.
\end{cases}
\]

\medskip
Then $(R, \oplus, \otimes)$ is a $\mathcal{A}$-definable ring. The definably connected component 
\[
R^0 = \{0\} \times [0, a[
\]

\medskip
 has two complements in $(R, \oplus)$: $F_1 = \{(0, 0), (1, 0)\}$ and $F_2 = \{(0, 0), (1, a/2)\}$, but neither is a subring. 

Note that $R$ is an example of $(b)$ from \Cref{theo:fields} where the finite ring 
\[
\left \{ 
\begin{pmatrix}
0 & 0 & 0 \\
x & 0 & 0 \\
y & x & 0
\end{pmatrix}
 : x, y \in \Z_2
 \right \}
\]

\medskip
embeds in $(R, \oplus, \otimes)$ as a 2-torsion ideal with non-trivial intersection with $R^0 \se \Ann(R)$.  

This also provides an example of \Cref{prop:disconnected}, where $n = 2$, $\mtf(G)$ is trivial and $R^2 \se T_2$.
\end{ex}

  %======================================
 %======================================

\end{document}